\documentclass[11pt,twoside,english]{article}
\usepackage{amsmath}
\usepackage{amsfonts}
\usepackage{mathrsfs}
\usepackage{color}
\usepackage{ amsmath, amsfonts, amssymb, amsthm, amscd}
\usepackage[T1]{fontenc}
\usepackage[english]{babel}
\usepackage[noinfoline]{imsart}
\usepackage{graphicx}
\usepackage{stmaryrd}  
\usepackage{titletoc}  

\newtheorem{theorem}{Theorem}
\newtheorem{lemma}{Lemma}
\newtheorem{proposition}{Proposition}

\newtheorem{definition}{Definition}
\newtheorem{corollary}{Corollary}
\newtheorem{remark}{Remark}
\newtheorem{assumption}{Assumption}

\newcommand{\cC}{\ensuremath{\mathcal C}}

\newcommand{\be}{\begin{equation}}
\newcommand{\ee}{\end{equation}}
\newcommand{\beq}{\begin{eqnarray}}
\newcommand{\eeq}{\end{eqnarray}}

\newcommand{\eps}{\varepsilon}

\newcommand{\ced}{\end{proof}}

\newcommand{\red}{\color{red}}

 \allowdisplaybreaks[3]

\begin{document}
\begin{frontmatter}
\title{It{\^o}-Wentzell-Lions formulae for flows of full and conditional measures on semimartingales}
\date{}
\runtitle{}
\author{\fnms{Jisheng}
 \snm{LIU}\corref{}\ead[label=e1]{jsliu23@m.fudan.edu.cn}}
\address{School of Mathematical Sciences\\
Fudan University, China \\
\printead{e1}}

\author{\fnms{Jing}
 \snm{ZHANG}\corref{}\ead[label=e2]{zhang\_jing@fudan.edu.cn}}
\thankstext{T2}{Research is supported by National Key R\&D Program of China (2022YFA1006101), National Natural Science Foundation of China (12271103, 12031009) and Shanghai Science and Technology Commission Grant (21ZR140860).}
\address{School of Mathematical Sciences \\
Fudan University, China  \\
\printead{e2}}

\runauthor{J. Liu and J. Zhang}

\begin{abstract}
In this paper, we establish the It\^o-Wentzell-Lions formulae for flows of both full and conditional measures on general semimartingales. This generalizes the existing work on the It\^o-Lions formula for general semimartingales and the It\^o-Wentzell-Lions formula for It\^o processes.
The key technical components involve an appropriate approximation of random fields by cylindrical functions and localization techniques. 
Moreover, we present the specific formulae in two special cases, including It\^o-Wentzell-Lions formulae for time-space-measure-dependent functions and for functions driven by Poisson random measures.
\end{abstract}


\end{frontmatter}

%
%
%
%
%
%
%
%
%
%
%
%
%

\section{Introduction}
It{\^o}-Wentzell formula is a general It{\^o} formula extended from deterministic functions to random fields. It was first introduced by Wentzell in \cite{ventzel1965equations}, then was generalized in \cite{kunita1981some,kunita1990stochastic,ocone1989generalized}, etc. It is sometimes called Kunita-It{\^o} formula to emphasize Kunita's contribution to the generalization. 
This key result has been widely used in the study of stochastic partial differential equations (SPDEs) and backward stochastic partial differential equations (BSPDEs), including the well-posedness, numerical methods and applications, for instance, the stochastic control problems with random coefficients (see \cite{peng1992stochastic,qiu2018viscosity}), and the representation relationship between BSPDEs and forward-backward stochastic differential equations (FBSDEs) with random coefficients (see \cite{hu2002semi} for example).\\
Recently, motivated by mean-field control problems and mean-field games, the theory of McKean-Vlasov stochastic differential equations (SDEs) has been extensively developed (see \cite{bensoussan2013mean,carmona2018probabilistic} and the references therein).  In these models, a new variable defined in the measure space, representing the distribution of the solution, is introduced.  To derive the Hamilton-Jacobi-Bellman equations for mean-field control problems (see \cite{cheung2025viscosity, pham2017dynamic}) and the master equations for mean-field games (see \cite{cardaliaguet2019master, carmona2018probabilistic, mou2019wellposedness}), the It\^o formula has been extended to the flows of probability measures, resulting in what is known as the “It{\^o}-Lions formula”.
This kind of formula was first proposed by Buckdahn et al. in \cite{buckdahn2017mean}, and for more details, the readers are referred to \cite{cardaliaguet2010notes,carmona2018probabilistic} for instance. 
Furthermore, when considering the control/game problems with random coefficients, the so called "It\^o-Wentzell-Lions formula" will be needed.\\
In order to study the mean-field SDEs driven by Brownian motion and an independent Poisson random measure, Hao and Li \cite{HaoLi} first introduced the It{\^o}-Lions formula for jump-diffusion processes. This result was later generalized by Li \cite{li2018mean} to avoid the need for the existence of the second-order mixed derivatives.
Recently, the It{\^o}-Lions formula along flows of probability measures associated with general semimartingales was established simultaneously by Guo, Pham and Wei \cite{guo2023ito} and Talbi, Touzi and Zhang \cite{talbi2023dynamic}. In \cite{guo2023ito}, they first proved the It\^o-Lions formula for cylindrical functions, whose derivatives with respect to measure are easy to compute, and then extended it to the general case via function approximation and localization techniques. In \cite{talbi2023dynamic}, the same formula was proved by using a more straightforward method, but under slightly different conditions. For details, see Remark 3.14 in \cite{guo2023ito} or Remark 3.3 in \cite{talbi2023dynamic}.
Afterwards, motivated by the mean-field systems with common noise, Guo and Zhang \cite{guo2024ito} extended the above It{\^o}-Lions formula to the case of conditional probability measures under a common filtration. Compared to the It{\^o}-Lions formula for flows of measures on diffusion processes, their new formulae introduce terms involving linear derivatives on the probability measure space. Such derivatives can characterize the behavior of both the jumps in the flows of probability measures and the jumps of the semimartingale processes.\\
In \cite{dos2023ito}, Dos Reis and Platonov constructed It{\^o}-Wentzell-Lions formulae for flows of both full and conditional  measures on continuous semimartingales driven by Brownian motions. They started with the standard It{\^o}-Wentzell formula and used the projection over empirical measures approach (as stated in \cite{carmona2018probabilistic} and Section 3 in \cite{chassagneux2014probabilistic}). The key point of this method is to approximate the distribution $\mu$ of a random variable $X$ by averaging the Dirac distributions of $N$ independent and identically distributed random variables, which allows us to approximate the It\^o-Wentzell-Lions formula by the classical It\^o-Wentzell formula. They chose the empirical projection technique rather than the cylindrical function technique because of the lack of a general Leibniz rule for interchanging the Fr{\'e}chet derivatives and the stochastic integral, which is one of the contributions of our present work. Fadle and Nizar \cite{FadleNizar} considered the state process driven by general continuous semimartingales with general conditioning. Their method follows standard arguments of It\^o calculus, combining the standard It\^o's formula and the standard limiting argument based on the dominated convergence theorem.\\
In this paper, we establish the It\^o-Wentzell-Lions formulae for flows of both full and conditional probability measures associated with general semimartingales by adopting the cylindrical technique in our framework. 
Our strategy is to first establish the It\^o-Wentzell-Lions formula for cylindrical functions based on the classical It\^o-Wentzell formula with general semimartingales, and then extend the results to the general case via approximation argument. However, this approach runs into difficulties: a direct approximation of the random fields by cylindrical functions poses significant measurability challenges.
Precisely, given that the approximation by cylindrical functions is established for deterministic univariate functions with respect to a measure, we must consider approximating functions that converge $(t,\omega)$-pointwisely. However, these approximating functions are no longer measurable with respect to the time variable and consequently lose adaptability. This loss of adaptability renders the It\^o-Wentzell-Lions formula inapplicable.
To overcome this difficulty, we approximate the random fields through a three-stage procedure: first, approximate the random fields by deterministic functions in an $N$-dimensional space with truncated semimartingales; then, approximate these deterministic functions with simple functions in time and space; and finally, approximate these simple functions by cylindrical functions. This procedure ensures the adaptedness of the approximating sequence.\\
We also derive explicit expressions of the It\^o-Wentzell-Lions formula for two special cases. The first is for random time-space-measure-dependent functions, a result motivated by McKean-Vlasov control problems where the coefficients depend on both the state trajectory and its distribution. Secondly, we address the case where both the state process and the random field are driven by Poisson random measures, as this is the most common jump-diffusion setting in stochastic control and game problems (see \cite{meng2023optimal, oksendal2019stochastic}).\\
Finally, we explain that the cylindrical function technique is suitable for our problem. First, the empirical projection method, which is most widely used in the mean-field framework, involves high computational complexity for the approximation. Whether it remains valid for general jump semimartingales requires further investigation. Second, the time discretization approach, which is more natural when considered by analogy with the classical case, was proposed by Wu and Zhang \cite{WuZhang} to establish the functional It\^o formula in the McKean-Vlasov setting. This method may prove applicable to our problem, though under slightly different assumptions, and we leave it for future studies.
In conclusion, our approach is selected for its computational simplicity and robust applicability to general semimartingales.\\
The rest of this paper is organized as follows. In Section 2, we set the notations and recall the concepts and properties related to derivatives of functions with respect to measures. We then prove the It{\^o}-Wentzell-Lions formula for full and conditional flows of measures on semimartingales in Section 3 and Section 4, respectively. The fifth section is devoted to presenting the specific forms of formulae in two special cases. Finally, some useful lemmas are proven in the appendix, including the Leibniz rule for interchanging the Fr\'echet derives and the It\^o integral.

\section{Preliminaries}

\subsection{Notations}
Let $(\Omega,\mathscr{F},\mathbb{F}:=(\mathscr{F}_t)_{t\geq0},\mathbb{P})$ be a standard filtered probability space satisfying the usual conditions. We assume that $\mathscr{F}$ is rich enough such that any sub-$\sigma$-algebras we mention are contained in it. Given a $\sigma$-algebra $\mathscr{G}\subset\mathscr{F}$,
$\mathbb{E}[\cdot|\mathscr{G}]$ denotes the conditional expectation on $\mathscr{G}$. For any $A\in \mathscr{F},\ \mathbb{P}(A|\mathscr{G})=\mathbb{E}[\mathbf{1}_A|\mathscr{G}]$ is the conditional probability on $\mathscr{G}$. The distribution of a random variable $X$ is denoted as $\mathbb{P}_X$.\\
Let $\mathcal{P}(\mathbb{R}^d)$ denote the set of all probability measures on $\mathbb{R}^d$, and the $p$-Wasserstein space of measures on $\mathbb{R}^d$ is defined as $$\mathcal{P}_p(\mathbb{R}^d):=\left\{\mu\in\mathcal{P}(\mathbb{R}^d)\Big|\left\lVert \mu\right\rVert_p:=\Big(\int_{\mathbb{R}^d} \,\mu(dx )\Big)^{\frac{1}{p}} <\infty \right\},$$ 
equipped with the distance $$W_p(\mu,\nu):=\inf_{\pi\in\Pi_{\mu,\nu}}\bigg(\int_{\mathbb{R}^d\times\mathbb{R}^d}  \,\left\lvert x_1-x_2\right\rvert^p \pi(dx_1,dx_2)\bigg)^{\frac{1}{p}},\quad \forall\ \mu,\ \nu\in \mathcal{P}_p(\mathbb{R}^d),$$ 
where $\Pi_{\mu,\nu}$ is the set of all the transport from $\mu$ to $\nu$. We can also see   $\Pi_{\mu,\nu}$ as a set of measures in $\mathcal{P}(\mathbb{R}^d\times\mathbb{R}^d)$ (the set of all joint probability measure on $\mathbb{R}^d\times\mathbb{R}^d$) with marginals $\mu$ and $\nu$.\\
For vectors $a,b$ in $\mathbb{R}^d$, denote $ab=a\cdot b:=\sum_{i=1}^d a_i b_i$, where $a_i,b_i$ are the $i$th components of $a$ and $b$, respectively. For matrix $A,A'\in \mathbb{R}^{d\times d}$, the trace of $A$ is defined by 
$Tr(A):=\sum_{i=1}^d A_{i,i}$. We write $Tr(A\cdot A')$ as $A:A'$ for simplicity.\\
For $k\in\mathbb{N}_0$, we denote by $C^k(\mathbb{R}^d)$ the space of all $k-$th differentiable functions on $\mathbb{R}^d$ with continuous derivatives up to $k-$th order and $C^k_b(\mathbb{R}^d)\subset C^k(\mathbb{R}^d)$ the space of functions with the up to $k-$th order derivatives being bounded and continuous on $\mathbb{R}^d$. When $k=0$, write simply $C(\mathbb{R}^d)$ and $C_b(\mathbb{R}^d)$. \\
The notation $\sharp$ before a set denotes the number of members of a finite set. If it appears as $\phi\sharp\mu$ with $\phi:\mathbb{R}^d\to \mathbb{R}^d$ a Borel measurable map and $\mu$ a probability measure on $\mathbb{R}^d$, then it denotes the image of $\mu$ by $\phi$ and is a probability measure satisfying
\begin{eqnarray*}
\int_{\mathbb{R}^d} \,f(x)\phi\sharp\mu(dx)=\int_{\mathbb{R}^d} \,f(\phi(y))\mu(dy),\quad \forall f\in C_b(\mathbb{R}^d).
\end{eqnarray*}
When we say $\tilde{X}$ is a copy of the random variable $X$ defined on $(\Omega,\mathscr{F},\mathbb{P})$, it means that $\tilde{X}$ is a lifted random variable defined on twin stochastic spaces $(\tilde{\Omega},\tilde{\mathscr{F}},\tilde{\mathbb{P}})$, having the same law as $X$. In other words, we can form a new probability space $(\Omega,\mathscr{F},\mathbb{P})\otimes(\tilde{\Omega},\tilde{\mathscr{F}},\tilde{\mathbb{P}})$ and consider the variable $\tilde{X}(\omega,\tilde{\omega}):=X'(\tilde{w})$, where $X'$ defined on $(\tilde{\Omega},\tilde{\mathscr{F}},\tilde{\mathbb{P}})$ has the same law as $X$. $\tilde{\mathbb{E}}[\cdot]$ denotes the expectation acting on the model twin space $\tilde{\Omega}$.

\subsection{Differentiability of functions of probability measures}
In this paper, we will use two types of derivatives of functions with respect to probability measures. One is the Lions derivative defined by the Fr{\'e}chet derivative on a Hilbert space, the other is the linear derivative defined directly on $\mathcal{P}_2(\mathbb{R}^d)$. In this subsection, we recall the definitions and propositions that will be used. For details, the readers are referred to, for example, \cite{cardaliaguet2010notes,cardaliaguet2019master,carmona2018probabilistic,gangbo2019differentiability}.\\
We denote by $L^2(\Omega;\mathbb{R}^d)$ the set of all $\mathbb{R}^d$-valued square-integrable random variables on $(\Omega,\mathscr{F},\mathbb{P})$ satisfying $\left\| X\right\|_{L^2}:=\left[\mathbb{E}|X|^2\right]^{\frac{1}{2}}<\infty$. \\
If $\mathscr{F}$ is rich enough to be atomless (i.e. for any $E\in \mathscr{F}$ with $\mathbb{P}(E)>0$, there exists $E'\in \mathscr{F}$ with $E'\subset E$ and $0<\mathbb{P}(E{'})<\mathbb{P}(E)$), for any map $U:\mathcal{P}_2(\mathbb{R}^d)\to \mathbb{R}$, 
we consider a lift $\widetilde{U} $ defined on $L^2(\Omega;\mathbb{R}^d)$ as following
\begin{eqnarray*}
\widetilde{U}(X):=U(\mathbb{P}_{X}) ,\ \ \ \ X\in L^2(\Omega;\mathbb{R}^d).
\end{eqnarray*}
Note that since $X\in L^2(\Omega;\mathbb{R}^d)$, we have $\mathbb{P}_{X}\in \mathcal{P}_2(\mathbb{R}^d)$ by definition, which makes $\tilde{U}$ well-defined.\\
The main point to define the Lions derivative is that $L^2(\Omega;\mathbb{R}^d)$ is a Hilbert space, in which the notion of Fr{\'e}chet differentiability makes sense. Precisely, 
$\widetilde{U}$ is said to be Fr{\'e}chet differentiable at $X_0$ if and only if there exists a linear continuous mapping $D\widetilde{U}(X_0):L^2(\Omega;\mathbb{R}^d)\to \mathbb{R} $ such that
\begin{eqnarray*}
\widetilde{U}(X)- \widetilde{U}(X_0)=\mathbb{E}[D\widetilde{U}(X_0)(X-X_0)]+o(\left\lVert X-X_0\right\rVert_{L^2} ), \ \mbox{as}\ \left\lVert X-X_0\right\rVert_{L^2}\to 0.
\end{eqnarray*}
It is shown in \cite{carmona2018probabilistic} (Proposition 5.25) that the law of $D\widetilde{U}(X_0)$ depends on $X_0$ only via its law $\mathbb{P}_{X_0}$, and there exists a Borel function $h_{\mathbb{P}_{X_0}}:\mathbb{R}^d\to\mathbb{R}^d$ satisfying
\begin{eqnarray}\label{lionsderivative}
D\widetilde{U}(X_0)=h_{\mathbb{P}_{X_0}}(X_0).
\end{eqnarray}
\begin{definition}\label{Lionsderiv}
$U$ is said to be differentiable at $\mu_0:=\mathbb{P}_{X_0}\in \mathcal{P}_2(\mathbb{R}^d)$ if its lift $\widetilde{U} $ is Fr{\'e}chet differentiable at $X_0$. In this case, 
we define the function $h_{\mathbb{P}_{X_0}}$ shown in (\ref{lionsderivative}) as the Lions derivative of $U$ at $\mu_0$, denoted as $\partial_{\mu}U(\mu_0,\cdot)$.
\end{definition}
Now we introduce the linear derivative on Wasserstain space $\mathcal{P}_2(\mathbb{R}^d)$.
\begin{definition}\label{linearderiv}
We say that $U:\mathcal{P}_2(\mathbb{R}^d)\to \mathbb{R}$ is of class $\mathcal{C}^1(\mathcal{P}_2(\mathbb{R}^d))$ if there exists a jointly continuous and bounded map $\frac{\delta U}{\delta\mu}: \mathcal{P}_2(\mathbb{R}^d)\times\mathbb{R}^d\to \mathbb{R}$ such that
\begin{eqnarray*}
U(\mu')-U(\mu)=\int_{0}^{1}  \,\int_{\mathbb{R}^d} \,\frac{\delta U}{\delta\mu}((1-h)\mu+h\mu',y)(\mu'-\mu)dy dh, \quad \forall \mu,\mu'\in\mathcal{P}_2(\mathbb{R}^d).
\end{eqnarray*}
Then $\frac{\delta U}{\delta\mu}$ is called a linear derivative. Moreover, we adopt the normalization convention
\begin{eqnarray*}
\int_{\mathbb{R}^d} \,\frac{\delta U}{\delta\mu}(\mu,y)\mu(dy)=0,\quad\forall\mu\in\mathcal{P}_2(\mathbb{R}^d).
\end{eqnarray*}
\end{definition}
\begin{proposition}(\cite{cardaliaguet2019master}, Proposition 2.2.3)
Assume that $U$ is in $\mathcal{C}^1(\mathcal{P}_2(\mathbb{R}^d))$ and $\partial_y\frac{\delta U}{\delta\mu}(\mu,y)$ exists and is jointly continuous and bounded on $\mathcal{P}_2(\mathbb{R}^d)\times\mathbb{R}^d$. Then for any Borel measurable map $\phi:\mathbb{R}^d\to\mathbb{R}^d$ 
with at most linear growth, the map $s\to U((id_{\mathbb{R}^d}+s\phi)\sharp \mu)$ is differentiable at 0 and
\begin{eqnarray*}
\frac{d}{ds}U((id_{\mathbb{R}^d}+s\phi)\sharp \mu)_{|_{s=0}}=\int_{\mathbb{R}^d} \,\partial_y\frac{\delta U}{\delta\mu}(\mu,y)\phi(y)\mu(dy) ,
\end{eqnarray*}
where $id$ denotes the identity map.
\end{proposition}
We know from Theorem 1.20 in \cite{cardaliaguet2010notes} that the Lions derivative and the linear derivative (also called the flat derivative) are connected through the following relation:
\begin{eqnarray*}
\partial_{\mu}U(\mu,y)=\partial_y\frac{\delta U}{\delta\mu}(\mu,y).
\end{eqnarray*}
One can also define the higher order derivatives. 
\begin{definition}
For a fixed $y\in\mathbb{R}^d$, if the map $\mu\rightarrow\frac{\delta U}{\delta\mu}(\mu,y)$ is in $\mathcal{C}^1(\mathcal{P}_2(\mathbb{R}^d))$, then we say that $U$ is of class $\mathcal{C}^2(\mathcal{P}_2(\mathbb{R}^d))$ and denote by $\frac{\delta^2 U}{\delta\mu^2}$ its second order derivative. 
\end{definition}
\begin{definition}
$U$ is said to be twice Lions differentiable if the map $\partial_{\mu}U$ is Lions differentiable w.r.t. $\mu$ with $\partial_{\mu\mu}U:=\partial_{\mu}\partial_{\mu}U(\mu,y,y')$.
\end{definition}
The following proposition is important and will be frequently used in our work.
\begin{proposition}
For $U\in\mathcal{C}^2(\mathcal{P}_2(\mathbb{R}^d))$ of a linear form: $U(\mu)=\langle u,\mu\rangle :=\int_{\mathbb{R}^d} \,u(x)\mu(dx)$ and the derivatives of $u$ have linear growth, then
\begin{eqnarray*}&&
\frac{\delta U}{\delta\mu}(\mu,y)=u(y),\ \partial_{\mu}U(\mu,y)=u'(y),\ \partial_y\partial_{\mu}U(\mu,y)=u''(y).
\end{eqnarray*}
\end{proposition}

\section{It{\^o}-Wentzell-Lions formula for flows of measures on semimartingales}{\label{c1}}

Consider two semimartingales $X$ and $Y$ taking values in  $\mathbb{R}^d$ and $\mathbb{R}^l$ respectively, whose decompositions are as follows,
\begin{equation}\label{decompositionsemimart}\forall t\in[0,T],\ \  X_t=X_0+M_t+V_t,\ \ Y_t=Y_0+N_t+U_t,\end{equation}
where the initial data $X_0,Y_0$ are deterministic, $M,N$ are continuous local martingales and $V,U$ are c{\`a}dl{\`a}g processes of finite variation.
Note that there is no need to assume that $X$ and $Y$ are independent. $X$ is the state  process and $Y$ derives the stochastic function $F$ in the following form 
\begin{eqnarray}{\label{decompositionF}}F(t,x)=F_0(x)+\int_{0}^{t}  \,G_s(x)ds+\int_{0}^{t}  \,H_s(x)dY_s,\quad t\in [0,T],\ x\in \mathbb{R}^d,\end{eqnarray}
where $\int_{0}^{t}H_s(x)dY_s:=\sum_{i=1}^{l}\int_{0}^{t}H_s^{(i)}(x)dY^{(i)}_s$, 
 $F_0:\mathbb{R}^d\to\mathbb{R}$ is deterministic, $G:\Omega\times[0,T]\times \mathbb{R}^d \mapsto \mathbb{R}$ is $\mathbb{F}-$progressively measurable and $H:\Omega\times[0,T]\times \mathbb{R}^d \mapsto \mathbb{R}^{l}$ is $\mathbb{F}-$predictable. 
\begin{definition}
We call the random field $F:\Omega\times[0,T]\times\mathbb{R}^d\rightarrow\mathbb{R}$ defined in (\ref{decompositionF}) RF-$C^2$, if $F_0(\cdot)$ is twice continuously differentiable w.r.t. $x$, so are $G_t(\omega,\cdot)$ and $H_t(\omega,\cdot)$ for any $(\omega,t)\in\Omega\times[0,T]$, and there is a constant $C$ such that 
\begin{equation*}\begin{split}
 \sup_{x\in \mathbb{R}^d}\left\lvert F_0(x)\right\rvert + \sup_{x\in \mathbb{R}^d}\left\lvert\partial_x F_0(x)\right\rvert+\sup_{x\in \mathbb{R}^d}\left\lvert\partial_{xx} F_0(x)\right\rvert &\leq C;\\
 \int_{0}^{T} \, \sup_{x\in \mathbb{R}^d}\left\lvert K_s(x)\right\rvert^2ds +\int_{0}^{T} \, \sup_{x\in \mathbb{R}^d}\left\lvert \partial_x K_s(x)\right\rvert^2ds+\int_{0}^{T} \, \sup_{x\in \mathbb{R}^d}\left\lvert \partial_{xx} K_s(x)\right\rvert^2ds &\leq C,\ \ \ a.s.,
\end{split}\end{equation*}
with $K:=G,H$.
\end{definition}
\begin{definition}\label{def RF-P-C2}
Consider the random field $F:\Omega\times[0,T]\times\mathcal{P}_2(\mathbb{R}^d)\rightarrow\mathbb{R}$ defined  by replacing $x$ with $\mu$ in \eqref{decompositionF}, i.e. 
\begin{eqnarray}\label{decompositionFmu}
F(t,\mu)=F_0(\mu)+\int_{0}^{t}G_s(\mu)ds+\int_{0}^{t}H_s(\mu)dY_s,\quad t\in [0,T],\ \mu\in\mathcal{P}_2(\mathbb{R}^d).\end{eqnarray}
$F$ is called $RF-Partially-\mathcal{C}^2$, if for any $(\omega,t)\in\Omega\times[0,T]$, $F_0(\cdot),G_t(\omega,\cdot),H_t(\omega,\cdot)$ are differentiable in $\mu$, and their derivatives $\partial_\mu F_0(\mu,\cdot),\partial_\mu G_t(\omega,\mu,\cdot),\partial_\mu H_t(\omega,\mu,\cdot)$ are differentiable in $x$ with continuous derivatives (note that, by the Leibniz rule proved in Appendix Lemma \ref{lemma Leibniz}, the Lions differentiability w.r.t. $\mu$ of $F_0,G,H$ implies the Lions differentiability of $F$). Moreover, there is a constant $C$ such that
\begin{align*}
 \sup_{x\in\mathbb{R}^d, \mu\in \mathcal{P}_2(\mathbb{R}^d)}\left\lvert F_0(\mu)\right\rvert +& \sup_{\mu\in \mathcal{P}_2(\mathbb{R}^d),x\in \mathbb{R}^d}\left\lvert\partial_{\mu} F_0(\mu,x)\right\rvert\\
+&\sup_{\mu\in \mathcal{P}_2(\mathbb{R}^d),x\in \mathbb{R}^d}\left\lvert\partial_{x}\partial_{\mu}F_0(\mu,x) \right\rvert  \leq C;\\
\int_{0}^{T}\sup_{\mu\in \mathcal{P}_2(\mathbb{R}^d)}\left\lvert K_s(\mu)\right\rvert^2ds +&\int_{0}^{T} \, \sup_{\mu\in \mathcal{P}_2(\mathbb{R}^d),x\in \mathbb{R}^d}\left\lvert \partial_{\mu}K_s(\mu,x)\right\rvert^2ds\\
+&\int_{0}^{T} \, \sup_{\mu\in \mathcal{P}_2(\mathbb{R}^d),x\in \mathbb{R}^d}\left\lvert \partial_x\partial_{\mu}K_s(\mu,x)\right\rvert^2ds  \leq C,\ \ \ a.s.,
\end{align*}
with $K:=G,H$. 
\end{definition}
\begin{remark}
The condition of boundedness in Definition \ref{def RF-P-C2} can be relaxed to that of polynomial growth w.r.t. variable $x$, i.e.
\begin{align*}
\sup_{\mu\in \mathcal{P}_2(\mathbb{R}^d),x\in \mathbb{R}^d}\left\lvert\partial_{\mu}F_0(\mu,x) \right\rvert+\sup_{\mu\in \mathcal{P}_2(\mathbb{R}^d),x\in \mathbb{R}^d}\left\lvert\partial_x\partial_{\mu}F_0(\mu,x) \right\rvert  \leq C(1+|x|^\alpha)
\end{align*}
for some $\alpha>0$, and the same for $G,H$. Then the estimation in our localization argument (see Step 2 in the proof of Theorem \ref{itowentzellc1thm}) should be replaced with that in Section 4.2 of \cite{guo2024ito}. In this paper, we assume the boundedness only for computational simplicity.
\end{remark}

We define the following assumptions.
\begin{assumption}\label{assumptionX}
For semimartingale $X$ given by \eqref{decompositionsemimart}, it holds that
\begin{eqnarray*}
\mathbb{E}\left[ Var(V)_T \right]<\infty,\ \mathbb{E}\bigg[ \sum_{0<t\leq T}\left\lvert \Delta V_t\right\rvert  \bigg]<\infty,\ \mathbb{E}\big[ [M,M]_T \big]<\infty.
\end{eqnarray*}
Here for process $\{P_t\}_{t\in[0,T]}$, $Var(P)_T$ denotes the variation of $P$ on $[0,T]$, $P_{t-}$ the left limit of $P$ at $t$, 
$\Delta P_t:=P_t-P_{t-}$ the jump of $P$ at $t$, $[P,P]_\cdot$ the quadratic variation of $P$.\\ 
The semimartingale $Y$ in \eqref{decompositionsemimart} satisfies the same conditions as $X$. Furthermore, the c{\`a}dl{\`a}g process $U$ has no jump at time $T$.
\end{assumption}

\begin{remark}
Under Assumption 1, any $RF-Partially-\mathcal{C}^2$ function $F$ satisfies the following boundedness
\begin{align*}
\sup_{t\in [0,T],\mu\in \mathcal{P}_2(\mathbb{R}^d)}\mathbb{E}\left\lvert F(t,\mu)\right\rvert^2 +& \sup_{t\in [0,T],\mu\in \mathcal{P}_2(\mathbb{R}^d),x\in \mathbb{R}^d}\mathbb{E}\left\lvert\partial_{\mu} F(t,\mu,x)\right\rvert^2\\
+&\sup_{t\in [0,T],\mu\in \mathcal{P}_2(\mathbb{R}^d),x\in \mathbb{R}^d}\mathbb{E}\left\lvert\partial_{x}\partial_{\mu}F(t,\mu,x) \right\rvert^2 < \infty.
\end{align*}
In fact, it follows \eqref{decompositionFmu} that
\begin{align*}
\mathbb{E}|F(t,\mu)|^2\leq 3\mathbb{E}\bigg[|F_0(\mu)|^2+\Big|\int_0^t\, G_s(\mu)ds\Big|^2+\Big|\int_0^t\, H_s(\mu)dY_s\Big|^2\bigg].
\end{align*}
Then by Cauchy-Schwarz inequality, we have 
\begin{align*}
\bigg|\int_0^t\, G_s(\mu)ds\bigg|^2\leq \bigg(\int_0^t\,|G_s(\mu)|^2ds\bigg)t\leq CT,\ \ \ a.s.
\end{align*}
Note the expression of $Y$, we can deduce
\begin{align*}
\mathbb{E}\bigg[\Big|\int_0^t\, H_s(\mu)dY_s\Big|^2\bigg]\leq 2\mathbb{E}\bigg[\Big|\int_0^t\, H_s(\mu)dU_s\Big|^2\bigg]+2\mathbb{E}\bigg[\Big|\int_0^t\, H_s(\mu)dN_s\Big|^2\bigg].
\end{align*}
Still by Cauchy-Schwarz inequality, it holds
\begin{align*}
\mathbb{E}\bigg[\Big|\int_0^t\, H_s(\mu)dU_s\Big|^2\bigg]\leq \mathbb{E}\bigg[\Big(\int_0^t\,|H_s(\mu)|^2dU_s\Big)Var(U)_t\bigg]\leq C,
\end{align*}
where the last inequality is obtained by Assumption \ref{assumptionX} for $U$ and the $L^2$-integrability of $H$. \\Finally, by It{\^o} isometry, we know that
\begin{align*}
\mathbb{E}\bigg[\Big|\int_0^tH_s(\mu)dN_s\Big|^2\bigg]=\mathbb{E}\bigg[\int_0^t|H_s(\mu)|^2d[N,N]_s\bigg],
\end{align*}
which is also bounded. The results on the derivatives of $F$ can be obtained similarly.
\end{remark}

To construct the It{\^o}-Wentzell-Lions formula for cylindrical functions, we need the following standard It{\^o} formula for general semimartingales.
\begin{theorem}\label{standitosemimart}
    For $g\in C^2(\mathbb{R}^d)$ and $X$ satisfying Assumption 1, it holds almost surely that, for any $0\leq s\leq t\leq T$,
    \begin{align*}
    g(X_t)-g(X_s)=&\int_{s}^{t}  \,\partial_x g(X_{r-})dX_r
        +\int_{s}^{t}  \,\partial_{xx}g(X_{r-}):d[X,X]_r^c\\
        &+\sum_{s<r\leq t}\big\{g(X_r)-g(X_{r-})-\partial_x g(X_{r-})\Delta X_r\big\},\end{align*}
where $[X,X]_\cdot^c$ denotes the quadratic variation of the continuous part of $X$.
\end{theorem}
\begin{proof}
Applying the classical It{\^o} formula to the continuous process $g(X_t)-\sum_{0<r\leq t}\{g(X_r)-g(X_{r-})\}$, the desired result can be easily obtained.
\end{proof}

We then construct the It{\^o}-Wentzell formula for general semimartingales.
\begin{theorem}\label{standitowentzellsemimart}
    Given semimartingales $X$ and $Y$ satisfying Assumption 1 and a random field $F:\Omega\times[0,T]\times\mathbb{R}^d\rightarrow\mathbb{R}$ which is $RF-C^2$, then, for any $0\leq s\leq t\leq T$, the following expansion holds $\mathbb{P}-$almost surely,
    \begin{equation}\label{itowentzelljump}\begin{split}
    F(t,X_t)&-F(s,X_s)=\int_{s}^{t}  \,G_r(X_{r-})dr+\int_{s}^{t}  \,H_r(X_{r-})dY_r+\int_{s}^{t}  \,\partial_x F(r-,X_{r-})dX_r\\
    &+\frac{1}{2}\int_{s}^{t}  \,\partial_{xx}F(r-,X_{r-}) :d[X,X]_r^c+\int_{s}^{t}  \,\partial_x H_r(X_{r-}):d[X,Y]_r^c\\
    &+\sum_{s<r\leq t}\big\{F(r,X_r)-F(r-,X_{r-})-\partial_x F(r-,X_{r-})\Delta X_r-H_r(X_{r-})\Delta Y_r\big\},\end{split}\end{equation}
where 
\begin{align*}
\int_{s}^{t}  \,\partial_x H_r(X_{r-}):d[X,Y]_r:=\sum_{j=1}^l\sum_{i=1}^d\int_{s}^{t}  \,\partial_{x_i} H_r^{(j)}(X_{r-})d[X^{(i)},Y^{(j)}]_r.
\end{align*}
\end{theorem}
\begin{proof}
Apply the standard It{\^o}-Wentzell formula (see Theorem 1.1 in \cite{kunita1981some}) to the following continuous random field
\begin{align*}
F(t,X_t)-\sum_{0<r\leq t}\{F(r,X_r)-F(r-,X_{r-})\},
\end{align*}
and obtain
\begin{equation}\label{continuouspartitowentzell}\begin{split}
F(t,X_t)&-F(s,X_s)-\sum_{s<r\leq t}\{F(r,X_r)-F(r-,X_{r-})\}\\
=&\int_{s}^{t}  \,G_r(X_{r-})dr+\int_{s}^{t}  \,H_r(X_{r-})dY^c_r+\int_{s}^{t}  \,\partial_x F(r-,X_{r-})dX^c_r\\
&+\frac{1}{2}\int_{s}^{t}  \,\partial_{xx}F(r-,X_{r-}) :d[X,X]_r^c+\int_{s}^{t}  \,\partial_x H_r(X_{r-}):d[X,Y]_r^c.
\end{split}\end{equation}
Then the following relations 
\begin{align*}
\int_{s}^{t}  \,H_r(X_{r-})dY^c_r=\int_{s}^{t}  \,H_r(X_{r-})dY_r-\sum_{s<r\leq t}\{H_r(X_{r-})\Delta Y_r\}
\end{align*}
and 
\begin{align*}
\int_{s}^{t}  \,\partial_x F(r-,X_{r-})dX^c_r=\int_{s}^{t}  \,\partial_x F(r-,X_{r-})dX_r+\sum_{s<r\leq t}\{\partial_x F(r-,X_{r-})\Delta X_r\},
\end{align*}
together with \eqref{continuouspartitowentzell} yield the desired formula.
\end{proof}
\begin{remark}
From the proof of Theorem \ref{standitowentzellsemimart}, we know that formula \eqref{itowentzelljump} can be written as 
\begin{align*}
F(t,X_t)-F(s,X_s)=&\int_{s}^{t}  \,G_r(X_{r-})dr+\int_{s}^{t}  \,H_r(X_{r-})dY_r^c+\int_{s}^{t}  \,\partial_x F(r-,X_{r-})dX_r^c\nonumber\\
&+\frac{1}{2}\int_{s}^{t}  \,\partial_{xx}F(r-,X_{r-}) :d[X,X]_r^c+\int_{s}^{t}  \,\partial_x H_r(X_{r-}):d[X,Y]_r^c\nonumber\\
&+\sum_{s<r\leq t}\{F(r,X_r)-F(r-,X_{r-})\}, \ \ \ a.s.
\end{align*}
\end{remark}
Now, we are ready to construct It{\^o}-Wentzell-Lions formula for a stochastic function defined in \eqref{decompositionFmu}. 
The main idea is that, we first prove this formula for the cylindrical functions, and then extend it to the $RF-Partially-\mathcal{C}^2$ functions by approximation method.
\begin{definition}\label{cylindricalfun}
    $F:\mathcal{P}_2(\mathbb{R}^d)\to \mathbb{R}$ is called a $\cC^2-$cylindrical function if for some $n\in \mathbb{N}$,
    \begin{eqnarray*}
    F(\mu)=f(\langle \mu,g^{(1)}\rangle,\dots \langle \mu,g^{(n)}\rangle ),
    \end{eqnarray*}
     with $f\in C^2(\mathbb{R}^n)$, $g^{(i)}\in C^2_b(\mathbb{R}^d)$, $i=1,\cdots,n$ and $\langle \mu,g\rangle:=\int_{\mathbb{R}^d} \,g(x)\mu(dx)$. 
\end{definition}
From Section 4 of \cite{cox2024controlled}, we know that the cylindrical functions are dense in $\mathcal{C}^1(\mathcal{P}_2(\mathbb{R}^d))$.
\begin{proposition}\label{cylindricalapproximationprop}
    Let $F$ be in $\mathcal{C}^1(\mathcal{P}_2(\mathbb{R}^d))$, assume that $\partial_{\mu}F(\mu,x)$ is differentiable w.r.t $x$ and all the derivatives of $F$ are bounded. Then there exists a sequence of $\cC^2-$cylindrical functions $\{f^n\}_{n\in \mathbb{N}^+}$ such that the following convergence holds point wisely, 
    \begin{eqnarray*}
    (f^n,\frac{\delta f^n}{\delta \mu},\partial_{\mu}f^n,\partial_x\partial_{\mu}f^n)\to (F,\frac{\delta F}{\delta \mu},\partial_{\mu}F,\partial_x\partial_{\mu}F).
    \end{eqnarray*}
    Moreover, the functions $f^n$ satisfy  
    \begin{eqnarray*}
    \left\lvert f^n\right\rvert +\left\lvert \partial_{\mu}f^n\right\rvert +\left\lvert \partial_x\partial_{\mu}f^n \right\rvert \leq C.
    \end{eqnarray*}
\end{proposition}
The proof of the second part of Proposition \ref{cylindricalapproximationprop} is similar to that of Theorem 4.4 in \cite{cox2024controlled}, so we just give the main steps. Construct an operator $T_n:C_b(\mathbb{R}^d)\to C_b(\mathbb{R}^d)$ and its adjoint $T_n^*:\mathcal{P}_2(\mathbb{R}^d)\to\mathcal{P}_2(\mathbb{R}^d)$, such that $T^*_n\mu\to\mu$ and $T_n\varphi_n\to\varphi$, when $\varphi_n\to\varphi$. Define $f^n(\mu):=F(T^*_n\mu)$, for $F\in \mathcal{C}^1(\mathcal{P}_2(\mathbb{R}^d))$, the following identities hold true,
\begin{align*}
\frac{\delta f^n}{\delta\mu}(\mu,x)=T_n\bigg(\frac{\delta F}{\delta\mu}(T^*_n\mu,\cdot)\bigg)(x),\ \ \ T_n\partial_x\varphi(x)=\partial_xT_n\varphi(x).
\end{align*}
Then the boundedness follows from the fact
\begin{align*}
|T_n\varphi(x)|\leq \sup_{y\in\mathbb{R}^d}|\varphi(y)|,
\end{align*}
that is
\begin{align*}
\bigg|\frac{\delta f^n}{\delta\mu}(\mu,x)\bigg|\leq \sup_{y\in\mathbb{R}^d}\bigg|\frac{\delta F}{\delta\mu}(T^*_n\mu,y)\bigg|\leq C.
\end{align*}
To ensure the approximation of random field by deterministic functions, we need a further assumption on $F$ in \eqref{decompositionFmu}.
\begin{assumption}
In \eqref{decompositionFmu}, $G$ and $H$ are adapted to the filtration generated by  $N$ and $U$, so is $F$. In other words, define $\hat{\mathscr{F}}_t:=\sigma\{N_s,\ U_s|0\leq s\leq
 t\}$, $F,\ G$ and $H$ are $\{\hat{\mathscr{F}}_t\}_{t\in[0,T]}$-adapted.
\end{assumption}
We come to state our first main theorem.
\begin{theorem}\label{itowentzellc1thm}
    Given semimartingales $X$ and $Y$ satisfying Assumption 1, and $F$ defined in \eqref{decompositionFmu} which is $RF-Partially-\cC^2$ satisfying Assumption 2. Set $\mu_t:=\mathbb{P}_{X_t}$, $\forall t\in[0,T]$. Then we have for any $0\leq s\leq t \leq T$,
    \begin{align}\label{itowentzellc1}
    F(t,\mu_t)-&F(s,\mu_s)=\int_{s}^{t}  \,G_r(\mu_{r-})dr +\int_{s}^{t}  \,H_r(\mu_{r-})dY_r\nonumber\\&
    +\tilde{\mathbb{E}}\left[  \int_{s}^{t}  \,\partial_{\mu}F(r-,\mu_{r-},\tilde{X}_{r-}) d\tilde{X}_r+\frac{1}{2}\partial_x\partial_{\mu}F(r-,\mu_{r-},\tilde{X}_{r-}):d[\tilde{X},\tilde{X}]_r^c \right]\nonumber\\&
    +\sum_{s<r\leq t}\left\{F(r,\mu_r)-F(r-,\mu_{r-})-H_r(\mu_{r-})\Delta Y_r\right\}\nonumber\\&
    +\tilde{\mathbb{E}}\bigg[ \sum_{s<r\leq t}\left\{ (\frac{\delta F}{\delta\mu}(r-,\mu_{r-},\tilde{X}_r)-\frac{\delta F}{\delta\mu}(r-,\mu_{r-},\tilde{X}_{r-}))\mathbf{1}_{\{\mu_r=\mu_{r-}\}}\right\}\nonumber\\&\qquad
    -\sum_{s<r\leq t}\left\{ \partial_{\mu}F(r-,\mu_{r-},\tilde{X}_{r-})\Delta \tilde{X}_r\right\}\bigg],\quad a.s.,
    \end{align} 
where $\tilde{X}$ denotes an independent and identically distributed copy of $X$ on a copy $\tilde{\Omega}$ of $\Omega$.
\end{theorem}
The proof is divided into 3 steps.\\
{\bf{Step 1}} Construct It{\^o}-Wentzell-Lions formula for cylindrical functions.\\
Consider $F(t,\mu_t)=F_0(\mu_0)+\int_{0}^{t}G_s(\mu_s)ds+\int_{0}^{t}H_s(\mu_s)dY_s$ with $F_0(\cdot)$, $G_s(\omega,\cdot)$, $H_s(\omega,\cdot)$ cylindrical functions for any $(\omega,s)\in \Omega\times[0,T]$, which admit the following expressions
\begin{eqnarray*}
    F_0(\mu)=f_0(\langle \eta^F_1,\mu \rangle,\dots \langle \eta_{n_1}^F,\mu\rangle ),\\
    G_s(\omega,\mu)=g_s(\omega,\langle \eta_1^G,\mu\rangle,\dots \langle\eta_{n_2}^G,\mu \rangle ),\\
    H_s(\omega,\mu)=h_s(\omega,\langle \eta^H_1,\mu\rangle,\dots \langle\eta^H_{n_3},\mu \rangle ),
\end{eqnarray*}
with $f_0$ being deterministic, $g$ being $\{\hat{\mathscr{F}}_t\}_{t\in[0,T]}$-adapted and $h$ being $\{\hat{\mathscr{F}}_t\}_{t\in[0,T]}$-predictable.\\
For convenience, we set $\{\eta_1,\dots,\eta_N\}:=\{\eta^F_1,\dots,\eta^F_{n_1},\eta^G_1,\dots,\eta^G_{n_2},\eta^H_1,\dots,\eta^H_{n_3}\}$, hence $N=n_1+n_2+n_3$. We then extend $f_0(\cdot),g_s(\omega,\cdot),h_s(\omega,\cdot)$ to $\mathbb{R}^N$ in a manner that is consistent with their original domains, but for simplicity, we still denote them by $f_0(\cdot),g_s(\omega,\cdot),h_s(\omega,\cdot)$. By Definition \ref{cylindricalfun}, we know that $\eta_j\in C^2_b(\mathbb{R}^d)$, $\forall j$.
For any $t\in [0,T]$, define $Z^j_t:=\langle \eta_j,\mu_t\rangle =\mathbb{E}[\eta_j(X_t)]$ and $\mathbf{Z}_t:=(Z^1_t,\cdots,Z_t^N)$.
By applying Theorem \ref{standitosemimart} to $\eta_j(X_t)$, we have
\begin{align}\label{itoforeta}
\eta_j(X_t)-\eta_j(X_s)=&\int_{s}^{t}  \,\partial_x\eta_j(X_{r-})dX_r+\frac{1}{2}\int_{s}^{t}  \,\partial _{xx}\eta_j(X_{r-}):d[X,X]^c_r\nonumber\\
&+\sum_{s<r\leq t}\{\eta_j(X_r)-\eta_j(X_{r-})-\partial_x\eta_j(X_{r-})\Delta X_r\}, \ \ \ a.s.
\end{align}
In order to take expectation on both sides of \eqref{itoforeta}, we assume that $X$ is bounded almost surely, hence $Z^j$ is also bounded. Then taking expectation yields
\begin{align}\label{integralz}
Z_t^j-Z_s^j=&\,\mathbb{E}\bigg[\int_{s}^{t}  \,\partial_x\eta_j(X_{r-})dX_r+\frac{1}{2}\int_{s}^{t}  \,\partial _{xx}\eta_j(X_{r-}):d[X,X]^c_r\nonumber\\
&\quad+\sum_{s<r\leq t}\{\eta_j(X_r)-\eta_j(X_{r-})-\partial_x\eta_j(X_{r-})\Delta X_r\}\bigg].
\end{align}
Now, we view $F(t,\mu_t)$ as a function of $(t,\mathbf{Z}_t)$ denoted by $f(t,\mathbf{Z}_t)$, which is expressed as
\begin{eqnarray*}&&
f(t,Z_t^1,\dots,Z_t^N)=f_0(Z_0^1,\dots,Z_0^N)+\int_{0}^{t}g_r(Z_r^1,\dots,Z_r^N)dr +\int_{0}^{t}h_r(Z_r^1,\dots,Z_r^N)dY_r.
\end{eqnarray*}
Then, thanks to Theorem \ref{standitowentzellsemimart}, we obtain
\begin{equation}\label{itoforf}\begin{split} 
F(t,\mu_t)-&F(s,\mu_s)=f(t,\mathbf{Z}_t)-f(s,\mathbf{Z}_s)\\
=&\int_{s}^{t}  \,g_r(\mathbf{Z}_{r-})dr+\int_{s}^{t}  \,h_r(\mathbf{Z_{r-}})dY_r +\underbrace{\sum_j\int_{s}^{t}  \,\partial_{z_j}f(r-,\mathbf{Z}_{r-})dZ^j_r}_{I_1}\\
&+\underbrace{\frac{1}{2}\sum_{j,k}\int_{s}^{t}  \,\partial_{z_j z_k}f(r-,\mathbf{Z}_{r-})d[Z^j,Z^k]^c_r}_{I_2} +\underbrace{\sum_j \int_{s}^{t}  \,\partial_{z_j}h_r(\mathbf{Z}_{r-})d[Z^j,Y]^c_r}_{I_3}\\
&+\underbrace{\sum_{s<r\leq t}\{f(r,\mathbf{Z}_r)-f(r-,\mathbf{Z}_{r-})-\sum_j \partial_{z_j}f(r-,\mathbf{Z}_{r-})\Delta Z_r^j-h_r(\mathbf{Z}_{r-})\Delta Y_r\}}_{I_4}.
\end{split}\end{equation}
For $I_1$, since
 $Z^j=\mathbb{E}[\eta_j(X)]$ is a deterministic process and $X$ is of finite variation, then $Z^j$ is a c{\`a}dl{\`a}g process of finite variation. Also note that $Y$ is a c{\`a}dl{\`a}g stochastic process of finite variation. Hence, we may take a partition $\pi_{s,t}^m:=\{s=t_0^m<\cdots<t^m_m=t\},m\in \mathbb{N}$, such that $\max_k\left\lvert t^m_k-t^m_{k+1}\right\rvert \to 0$, $\max_k \sup_{r\in (t^m_k,t^m_{k+1}]} \left\lvert \mathbf{Z}_{r-}-\mathbf{Z}_{t^m_{k}}\right\rvert\to 0 $, and $\max_k \sup_{r\in (t^m_k,t^m_{k+1}]} \left\lvert Y_{r-}-Y_{t^m_{k}}\right\rvert\to 0$, a.s.
Then, by the definition of integral, 
\begin{eqnarray*}
I_1=\sum_j\int_{s}^{t}  \,\partial_{z_j}f(r-,\mathbf{Z}_{r-})dZ^j_r=\lim_{m\to \infty}\sum_{\pi_{s,t}}\sum_j\partial_{z_j}f(t^m_k,\mathbf{Z}_{t^m_k})(Z^j_{t^m_{k+1}}-Z^j_{t^m_k}).
\end{eqnarray*}
Combining the above formula with (\ref{integralz}) and Funibi's theorem yields
\begin{align*}
I_1=&\lim_{m\to \infty}\sum_{k=0}^m \tilde{\mathbb{E}}\bigg[ \int_{t^m_k}^{t^m_{k+1}}  \,\sum_{j}\partial_{z_j}f(t_k^m,\mathbf{Z}_{t^m_k})\partial_x\eta_j(\tilde{X}_{r-}) d\tilde{X}_r\\
 &\qquad\qquad\qquad+\frac{1}{2}\int_{t^m_k}^{t^m_{k+1}}  \,\sum_j \partial_{z_j}f(t_k^m,\mathbf{Z}_{t^m_k})\partial_{xx}\eta_j(\tilde{X}_{r-}):d[\tilde{X},\tilde{X}]^c_r \\
&\qquad\qquad\qquad+\partial_{z_j}f(t_k^m,\mathbf{Z}_{t^m_k}) \sum_{t_k^m<r\leq t^m_{k+1}}\{\eta_j(\tilde{X}_r)-\eta_j(\tilde{X}_{r-})-\partial_x\eta(\tilde{X}_{r-})\Delta \tilde{X}_r\} \bigg].
\end{align*}
Then, by taking limit $m\to\infty$, we deduce 
\begin{align*}
I_1=&\tilde{\mathbb{E}}\bigg[ \int_{s}^{t}\, \sum_{j}\partial_{z_j}f(r-,\mathbf{Z}_{r-})\partial_x\eta_j(\tilde{X}_{r-})d\tilde{X}_r\\
&\qquad+\frac{1}{2}\int_{s}^{t}\, \sum_{j}\partial_{z_j}f(r-,\mathbf{Z}_{r-})\partial_{xx}\eta_j(\tilde{X}_{r-}):d[\tilde{X},\tilde{X}]^c_r \\
&\qquad+\sum_{s<r\leq t}\{ \partial_{z_j}f(r-,\mathbf{Z}_{r-})(\eta_j(\tilde{X}_r)-\eta_j(\tilde{X}_{r-})-\partial_x\eta(\tilde{X}_{r-})\Delta \tilde{X}_r) \} \bigg].
\end{align*}
On the other hand, a direct calculus provides
\begin{align}\label{interchangederivativewithitointegral}
\partial_{\mu}F(t,\mu_t,X_t)=&\,\partial_{\mu}F_0(\mu_t,X_t)+\int_{0}^{t}  \,\partial_{\mu}G_s(\mu_s,X_s)ds +\int_{0}^{t}  \,\partial_{\mu}H_s(\mu_s,X_s)dY_s\nonumber \\
=&\sum_j\partial_{z_j} f_0(\mathbf{Z}_t)\partial_x\eta_j(X_t)+\int_{0}^{t}  \,\sum_j\partial_{z_j} g_s(\mathbf{Z}_s)\partial_x\eta_j(X_s)ds\nonumber \\
&+\int_{0}^{t}  \,\sum_j\partial_{z_j} h_s(\mathbf{Z}_s)\partial_x\eta_j(X_s)dY_s\nonumber  \\
=&\sum_j\partial_{z_j} f(t,\mathbf{Z}_t)\partial_x\eta_j(X_t),
\end{align}
and
\begin{align}\label{directcalculus}
\partial_x\partial_{\mu}F(t,\mu_t,X_t)=\sum_j\partial_{z_j} f(t,\mathbf{Z}_t)\partial_{xx}\eta_j(X_t),\ \ \ 
\frac{\delta F}{\delta \mu}(t,\mu_t,X_t)=\sum_j\partial_{z_j} f(t,\mathbf{Z}_t)\eta_j(X_t).
\end{align}
Note that in order to get the first equality of \eqref{interchangederivativewithitointegral}, we interchange the Fr\'echet derivative with stochastic integral due to Lemma \ref{lemma Leibniz} in Appendix B.\\
Therefore, 
\begin{align*}
I_1=&\,\tilde{\mathbb{E}}\bigg[ \int_{s}^{t}  \,\partial_{\mu}F(r-,\mu_{r-},\tilde{X}_{r-})d\tilde{X}_r+\frac{1}{2}\int_{s}^{t}  \,\partial_x\partial_{\mu}F(r-,\mu_{r-},\tilde{X}_{r-}):d[\tilde{X},\tilde{X}]_r^c \\
&\quad+\sum_{s<r\leq t}\left\{ \frac{\delta F}{\delta \mu}(r-,\mu_{r-},\tilde{X}_{r})-\frac{\delta F}{\delta \mu}(r-,\mu_{r-},\tilde{X}_{r-}) -\partial_{\mu}F(r-,\mu_{r-},\tilde{X}_{r-})\Delta \tilde{X}_r\right\} \bigg].
\end{align*}
For $I_2$, note that $Z^j$ is deterministic and admits expression (\ref{integralz}), we have
\begin{align*}
d(Z^j_r)^c&=\mathbb{E} \Big[ \partial_x\eta_j(X_{r-})dX^c_r+\frac{1}{2}\partial_{xx}\eta_j(X_{r-}):d[X,X]^c_r\Big]\\
&=\mathbb{E}\Big[\partial_x\eta_j(X_{r-})dV^c_r+\frac{1}{2}\partial_{xx}\eta_j(X_{r-}):d[M,M]_r\Big].
\end{align*}
Then, by the finite variation of $V$, one has $[Z^j,Z^i]_r^c=0$, $\forall i,j$. Hence, $I_2=0$.\\
For $I_3$, we rewrite it as following
\begin{eqnarray*}
I_3=\sum_j \sum_{n=1}^l\int_{s}^{t}  \,\partial_{z_j}h_r^{(n)}(\mathbf{Z}_{r-})d[Z^j,Y^{(n)}]^c_r,
\end{eqnarray*}
where $(n)$ denotes the $n-$th component of $Y$.
Note again that $Z^j$ is deterministic, we obtain
\begin{equation*}
(dZ^j_r dY^{(n)}_r)^c=\tilde{\mathbb{E}}\big[\partial_x\eta_j(\tilde{X}_{r-})(d\tilde{V}_rdY_r^{(n)})^c+\frac{1}{2}\partial_{xx}\eta_j(\tilde{X}_{r-}):(d[\tilde{M},\tilde{M}]_rdY^{(n)}_r)^c\big],
\end{equation*}
therefore, $(dZ^j_r dY^{(n)}_r)^c=0$, which yields $I_3=0$.\\
For $I_4$, by \eqref{directcalculus} we have
\begin{align*}
I_4=&\sum_{s<r\leq t}\bigg\{ F(r,\mu_r)-F(r-,\mu_{r-})-h_r(\mathbf{Z}_{r-})(Y_r-Y_{r-})-\sum_j \partial_{z_j}f(r-,\mathbf{Z}_{r-})(Z^j_r-Z^j_{r-})\bigg\}\\
=&\sum_{s<r\leq t}\bigg\{ F(r,\mu_r)-F(r-,\mu_{r-}) -H(r-,\mu_{r-})\Delta Y_r \\
&\qquad\qquad-\tilde{\mathbb{E}}\left[ \frac{\delta F}{\delta\mu}(r-,\mu_{r-},\tilde{X}_{r})-\frac{\delta F}{\delta\mu}(r-,\mu_{r-},\tilde{X}_{r-}) \right] \bigg\}.
\end{align*}
The last term in the above equality can be written as
\begin{equation}\label{eqsum}\begin{split}
&\tilde{\mathbb{E}}\bigg[ \sum_{s<r\leq t}\left( \frac{\delta F}{\delta\mu}(r-,\mu_{r-},\tilde{X}_{r})-\frac{\delta F}{\delta\mu}(r-,\mu_{r-},\tilde{X}_{r-}) \right) \bigg]\\
=&\,\tilde{\mathbb{E}}\bigg[\sum_{s<r\leq t} \left( \frac{\delta F}{\delta\mu}(r-,\mu_{r-},\tilde{X}_{r})-\frac{\delta F}{\delta\mu}(r-,\mu_{r-},\tilde{X}_{r-}) \right)\mathbf{1}_{\{\mu_r\neq \mu_{r-}\}} \bigg]\\
&+\tilde{\mathbb{E}}\bigg[ \sum_{s<r\leq t}\left( \frac{\delta F}{\delta\mu}(r-,\mu_{r-},\tilde{X}_{r})-\frac{\delta F}{\delta\mu}(r-,\mu_{r-},\tilde{X}_{r-}) \right)\mathbf{1}_{\{\mu_r= \mu_{r-}\}} \bigg].
\end{split}\end{equation}
Since $\mu_\cdot$ is a $\mathcal{P}_2(\mathbb{R}^d)$-valued c{\`a}dl{\`a}g process and $\{r\in[s,t]\,|\,\mu_r\neq\mu_{r-}\}$ is a countable set, we can interchange the sum and expectation of the first term on the right hand side of \eqref{eqsum} and get
\begin{align}\label{eq sum explain1}
&\tilde{\mathbb{E}}\bigg[ \sum_{s<r\leq t}\left( \frac{\delta F}{\delta\mu}(r-,\mu_{r-},\tilde{X}_{r})-\frac{\delta F}{\delta\mu}(r-,\mu_{r-},\tilde{X}_{r-}) \right) \bigg]\nonumber\\
=&\sum_{s<r\leq t}\tilde{\mathbb{E}}\bigg[ \left( \frac{\delta F}{\delta\mu}(r-,\mu_{r-},\tilde{X}_{r})-\frac{\delta F}{\delta\mu}(r-,\mu_{r-},\tilde{X}_{r-}) \right)\mathbf{1}_{\{\mu_r\neq \mu_{r-}\}} \bigg]\\
&+\tilde{\mathbb{E}}\bigg[ \sum_{s<r\leq t}\left( \frac{\delta F}{\delta\mu}(r-,\mu_{r-},\tilde{X}_{r})-\frac{\delta F}{\delta\mu}(r-,\mu_{r-},\tilde{X}_{r-}) \right)\mathbf{1}_{\{\mu_r= \mu_{r-}\}} \bigg].\nonumber
\end{align}
Meanwhile, on the set $\{\mu_r=\mu_{r-}\}$, it holds $\tilde{\mathbb{E}}[ \frac{\delta F}{\delta\mu}(r-,\mu_{r-},\tilde{X}_{r})]=\tilde{\mathbb{E}}[\frac{\delta F}{\delta\mu}(r-,\mu_{r-},\tilde{X}_{r-})]$, hence
\begin{align}\label{eq sum explain2}&
\sum_{s<r\leq t}\tilde{\mathbb{E}}\left[ \left( \frac{\delta F}{\delta\mu}(r-,\mu_{r-},\tilde{X}_{r})-\frac{\delta F}{\delta\mu}(r-,\mu_{r-},\tilde{X}_{r-}) \right) \right]\nonumber\\=&\,
\tilde{\mathbb{E}}\bigg[ \sum_{s<r\leq t}\left( \frac{\delta F}{\delta\mu}(r-,\mu_{r-},\tilde{X}_{r})-\frac{\delta F}{\delta\mu}(r-,\mu_{r-},\tilde{X}_{r-}) \right)\mathbf{1}_{\{\mu_r\neq\mu_{r-}\}} \bigg].
\end{align}
So we conclude by subtracting \eqref{eq sum explain2} from \eqref{eq sum explain1}  that
\begin{align*}
&\tilde{\mathbb{E}}\bigg[ \sum_{s<r\leq t}\left( \frac{\delta F}{\delta\mu}(r-,\mu_{r-},\tilde{X}_{r})-\frac{\delta F}{\delta\mu}(r-,\mu_{r-},\tilde{X}_{r-}) \right) \bigg]\\
&\quad-\sum_{s<r\leq t}\mathbb{E}\left[ \left( \frac{\delta F}{\delta\mu}(r-,\mu_{r-},\tilde{X}_{r})-\frac{\delta F}{\delta\mu}(r-,\mu_{r-},\tilde{X}_{r-}) \right) \right]\\
=&\,\tilde{\mathbb{E}}\bigg[ \sum_{s<r\leq t}\left( \frac{\delta F}{\delta\mu}(r-,\mu_{r-},\tilde{X}_{r})-\frac{\delta F}{\delta\mu}(r-,\mu_{r-},\tilde{X}_{r-}) \right)\mathbf{1}_{\{\mu_r=\mu_{r-}\}} \bigg].
\end{align*}
Until now, we have proven formula (\ref{itowentzellc1}) for cylindrical functions with bounded $X$.
\\{\bf{Step 2}} Localization argument.\\
Take $\tau_n:=\inf\{t:X_t\geq n\}$. 
Due to the definition and the local boundedness of $X$, it is clear that $\tau_n\to T$, as $n\to\infty$.\\
Denote $(X_t^n,\mu_{t}^n,Y_{t}^n):=(X_{t\wedge \tau_n},\mu_{t\wedge \tau_n},Y_{t\wedge \tau_n})$, by Step 1, the following relation holds true
\begin{align}\label{itowentzellc1local}
F(t&\wedge \tau_n,\mu^n_t)-F(s\wedge \tau_n,\mu^n_s)=\int_{s}^{t}  \,G_r(\mu^n_{r-})dr +\int_{s}^{t}  \,H_r(\mu^n_{r-})dY^n_r\nonumber\\&
    +\tilde{\mathbb{E}}\left[\int_{s}^{t}  \,\partial_{\mu}F((r-)\wedge \tau_n,\mu^n_{r-},\tilde{X}^n_{r-}) d\tilde{X}^n_r\right]+\tilde{\mathbb{E}}\left[\frac{1}{2}\partial_x\partial_{\mu}F((r-)\wedge \tau_n,\mu^n_{r-},\tilde{X}^n_{r-}):d[\tilde{X}^n,\tilde{X}^n]_r^c \right]\nonumber\\&
    +\sum_{s<r\leq t}\left\{F(r\wedge \tau_n,\mu^n_r)-F((r-)\wedge \tau_n,\mu^n_{r-})-H_r(\mu^n_{r-})\Delta Y^n_r\right\}\nonumber\\&
    +\tilde{\mathbb{E}}\bigg[ \sum_{s<r\leq t}\left\{ (\frac{\delta F}{\delta\mu}((r-)\wedge \tau_n,\mu^n_{r-},\tilde{X}^n_r)-\frac{\delta F}{\delta\mu}((r-)\wedge \tau_n,\mu^n_{r-},\tilde{X}^n_{r-}))\mathbf{1}_{\{\mu^n_r=\mu^n_{r-}\}}\right\} \nonumber\\&\qquad
    -\sum_{s<r\leq t}\left\{ \partial_{\mu}F((r-)\wedge \tau_n,\mu^n_{r-},\tilde{X}^n_{r-})\Delta \tilde{X}^n_r\right\}\bigg].
\end{align}
For the first term on the right hand side of (\ref{itowentzellc1local}), it is obvious that
\begin{equation*}
\int_{s}^{t}G_r(\mu^n_{r-})dr=\int_{s}^{t}G_r(\mu_{r-})\mathbf{1}_{[0,\tau_n]}(r)dr.
\end{equation*}
By Cauchy-Schwarz inequality and the integrability of $G$, we obtain
\begin{align*}
\left\lvert \int_{s}^{t}G_r(\mu_{r-})\mathbf{1}_{[0,\tau_n]}dr\right\rvert &\leq \int_{s}^{t}\left\lvert G_r(\mu_{r-})\right\rvert \mathbf{1}_{[0,\tau_n]}dr
\leq \bigg(\int_{s}^{t}  \,\left\lvert G_r(\mu_{r-})\right\rvert^2 dr\bigg)^{\frac{1}{2}}(t-s)^{\frac{1}{2}}\\
&\leq C^{\frac{1}{2}}(t-s)^{\frac{1}{2}},\quad a.s.
\end{align*}
Hence, we deduce by the dominated convergence theorem that,
$$\lim_{n\to\infty}\int_{s}^{t}G_r(\mu^n_{r-})dr=\int_{s}^{t}G_r(\mu_{r-})dr,\quad a.s.$$\\
For the second term, it follows the stopping rule for stochastic integral that
\begin{equation*}
\int_{s}^{t}  \,H_r(\mu^n_{r-})dY^n_r=\int_{s}^{t}  \,H_r(\mu_{r-})\mathbf{1}_{[0,\tau_n]}dY_r=\int_{s}^{t}  \,H_r(\mu_{r-})\mathbf{1}_{[0,\tau_n]}dU_r+\int_{s}^{t}  \,H_r(\mu_{r-})\mathbf{1}_{[0,\tau_n]}dN_r.
\end{equation*}
Since $U$ is of finite variation, it holds $\lim_{n\to\infty}\int_{s}^{t}H_r(\mu^n_{r-})dU_r=\int_{s}^{t}H_r(\mu_{r-})dU_r$, a.s. For the stochastic term, by It{\^o} isometry, we have
\begin{align*}
\sup_{n}\mathbb{E}\bigg[ \Big\lvert \int_{s}^{t}  \,H_r(\mu_{r-})\mathbf{1}_{[0,\tau_n]} dN_r \Big\rvert^2  \bigg]
=&\sup_{n}\mathbb{E}\left[  \int_{s}^{t}  \,\left\lvert H_r(\mu_{r-})\mathbf{1}_{[0,\tau_n]}\right\rvert^2 \cdot d[Y,Y]_r   \right]\\
\leq& C\mathbb{E}\big[ [Y,Y]_t \big] < \infty,
\end{align*}
which implies the uniform integrability of $\{\int_{s}^{t}  \,H_r(\mu_{r-})\mathbf{1}_{[0,\tau_n]} dN_r\}_n$, hence
\begin{equation*}
\int_{s}^{t}  \,H_r(\mu_{r-})\mathbf{1}_{[0,\tau_n]} dN_r \to \int_{s}^{t}  \,H_r(\mu_{r-}) dN_r, \ \ \ \mbox{in} \ L^2(\Omega).
\end{equation*}
Then we can take a subsequence $\{\tau_{n_k}\}_{k\in\mathbb{N}}$ s.t. when $k\to\infty$,
\begin{equation*}
\int_{s}^{t}  \,H_r(\mu_{r-})\mathbf{1}_{[0,\tau_{n_k}]} dN_r \to \int_{s}^{t}  \,H_r(\mu_{r-}) dN_r, \ \ \ a.s.
\end{equation*}
and still denote $\{\tau_{n_k}\}_{k\in\mathbb{N}}$ as $\{\tau_n\}_{n\in\mathbb{N}}$.\\
For the third term, still by the stopping rule for stochastic integral, one has
\begin{eqnarray*}
\int_{s}^{t}  \,\partial_{\mu}F((r-)\wedge\tau_n,\mu_{r-}^n,\tilde{X}_{r-}^n) d\tilde{X}_r^n=\int_{s}^{t}  \,\partial_{\mu}F(r-,\mu_{r-},\tilde{X}_{r-})\mathbf{1}_{[0,\tau_n]} d\tilde{X}_r.
\end{eqnarray*}

Due to the expression of $X$, it holds
\begin{align}\label{eq step2 partial_mu F}
&\mathbb{E}\bigg[\bigg|\tilde{\mathbb{E}}\Big[\int_s^t\,(\partial_\mu F(r-,\mu_{r-},\tilde{X}_{r-})\mathbf{1}_{[0,\tau_n]}-\partial_\mu F(r-,\mu_{r-},\tilde{X}_{r-}))d\tilde{X}_r\Big]\bigg|^2\bigg]\nonumber\\
\leq &\,2\mathbb{E}\tilde{\mathbb{E}}\bigg[\Big|\int_s^t\, (\partial_\mu F\mathbf{1}_{[0,\tau_n]}-\partial_\mu F)d\tilde{V}_r\Big|^2+\Big|\int_s^t\, (\partial_\mu F\mathbf{1}_{[0,\tau_n]}-\partial_\mu F)d\tilde{M}_r\Big|^2\bigg],
\end{align}
here and after we omit the variables of the functions for convenience. By Fubini's Theorem and Cauchy-Schwarz inequality, we have
\begin{align*}
\mathbb{E}\tilde{\mathbb{E}}\bigg[\Big|\int_s^t\, (\partial_\mu F\mathbf{1}_{[0,\tau_n]}-\partial_\mu F)d\tilde{V}_r\Big|^2\bigg]\leq \tilde{\mathbb{E}}\bigg[\int_s^t\,\mathbb{E}|\partial_\mu F\mathbf{1}_{[0,\tau_n]}-\partial_\mu F|^2d\tilde{V}_r\cdot Var(\tilde{V})_{[s,t]}\bigg].
\end{align*}
From the fact that $\mathbb{E}[|\partial_\mu F(r-,\mu_{r-},X_{r-})\mathbf{1}_{[0,\tau_n]}|^2]\leq \sup_{t,\mu,x}\mathbb{E}[|\partial_\mu F(t,\mu,x)|^2]<\infty$, we know
\begin{align*}
\mathbb{E}\tilde{\mathbb{E}}\bigg[\Big|\int_s^t\, (\partial_\mu F\mathbf{1}_{[0,\tau_n]}-\partial_\mu F)d\tilde{V}_r\Big|^2\bigg]\leq 4\tilde{\mathbb{E}}\bigg[\int_s^t\,\sup_{t,\mu,x}\mathbb{E}|\partial_\mu F|^2d\tilde{V}_r\cdot Var(\tilde{V})_{[s,t]}\bigg]<\infty.
\end{align*}
Hence, by the dominated convergence theorem, we obtain
\begin{align*}
\mathbb{E}\tilde{\mathbb{E}}\bigg[\Big|\int_s^t\, (\partial_\mu F\mathbf{1}_{[0,\tau_n]}-\partial_\mu F)d\tilde{V}_r\Big|^2\bigg]\to 0,\quad\mbox{as}\ n\to\infty.
\end{align*}
Next, by Fubini's Theorem and It{\^o} isometry, one has
\begin{align*}
\mathbb{E}\tilde{\mathbb{E}}\bigg[\Big|\int_s^t\, (\partial_\mu F\mathbf{1}_{[0,\tau_n]}-\partial_\mu F)d\tilde{M}_r\Big|^2\bigg]= \tilde{\mathbb{E}}\bigg[\int_s^t\,\mathbb{E}|\partial_\mu F\mathbf{1}_{[0,\tau_n]}-\partial_\mu F|^2d[\tilde{M},\tilde{M}]_r\bigg].
\end{align*}
Hence, similarly to the above and by the dominated convergence theorem, we deduce 
\begin{align*}
\mathbb{E}\tilde{\mathbb{E}}\bigg[\Big|\int_s^t\, (\partial_\mu F\mathbf{1}_{[0,\tau_n]}-\partial_\mu F)d\tilde{M}_r\Big|^2\bigg]\to 0,\ \ {\text{as}} \ 
 n\to\infty.
\end{align*}
Combining the above convergence with \eqref{eq step2 partial_mu F} provides
\begin{align*}
\mathbb{E}\tilde{\mathbb{E}}\bigg[\Big|\int_s^t\, (\partial_\mu F\mathbf{1}_{[0,\tau_n]}-\partial_\mu F)d\tilde{X}_r\Big|^2\bigg]\to 0,\ \ {\text{as}} \ 
 n\to\infty.
\end{align*}
Then we take a subsequence of $\{\tau_n\}$ and still denotes it as $\{\tau_n\}$, such that
\begin{align*}
\tilde{\mathbb{E}}\bigg[\int_s^t\, (\partial_\mu F\mathbf{1}_{[0,\tau_n]}-\partial_\mu F)d\tilde{X}_r\bigg]\to 0,\quad a.s.
\end{align*}
For the fourth term, we can use the same argument as above.

For the fifth term, the part
$-\sum_{s<r\leq t}H_r(\mu^n_{r-})\Delta Y^n_r$ can be absorbed into the second term $\int_{s}^{t}H_r(\mu^n_{r-})dY^n_r$ to become $\int_{s}^{t}H_r(\mu^n_{r-})(dY^n_r)^c$, whose convergence can be deduced similarly to $\int_{s}^{t}H_r(\mu^n_{r-})dY^n_r$.
Meanwhile, the following convergence is clear by the stopping time rule for stochastic integral:
\begin{equation*}
\sum_{s<r\leq t}\{F(r\wedge \tau_n,\mu^n_r)-F((r-)\wedge \tau_n,\mu^n_{r-})\}\to \sum_{s<r\leq t}\{F(r,\mu_r)-F(r-,\mu_{r-})\},\ \ \ a.s.
\end{equation*}
For the last terms, we are going to prove that
\begin{eqnarray*}&&
\tilde{\mathbb{E}}\sum_{s<r\leq t}\left\{ (\frac{\delta F}{\delta\mu}((r-)\wedge \tau_n,\mu^n_{r-},\tilde{X}_r^n)-\frac{\delta F}{\delta\mu}((r-)\wedge \tau_n,\mu^n_{r-},\tilde{X}^n_{r-}))\mathbf{1}_{\mu_r=\mu_{r-}}\right\}\\&&
\to \tilde{\mathbb{E}}\sum_{s<r\leq t}\left\{ (\frac{\delta F}{\delta\mu}(r-,\mu_{r-},\tilde{X}_r)-\frac{\delta F}{\delta\mu}(r-,\mu_{r-},\tilde{X}_{r-}))\mathbf{1}_{\mu_r=\mu_{r-}}\right\},\ \ \ a.s.
\end{eqnarray*}
From Remark 3.2 in \cite{guo2023ito}, we know that
\begin{align*}
&\sum_{s<r\leq t}(\frac{\delta F}{\delta\mu}((r-)\wedge \tau_n,\mu^n_{r-},\tilde{X}_r^n)-\frac{\delta F}{\delta\mu}((r-)\wedge \tau_n,\mu^n_{r-},\tilde{X}^n_{r-}))\mathbf{1}_{\mu_r=\mu_{r-}}\\
=&\sum_{s<r\leq t}\int_0^1\, \partial_\mu F(r-,\mu_{r-},\tilde{X}_{r-}+h\Delta \tilde{X}_r)\mathbf{1}_{[0,\tau_n]}\cdot \Delta \tilde{X}dh\\
=&\int_s^t\,\int_0^1\, \partial_\mu F(r-,\mu_{r-},\tilde{X}_{r-}+h\Delta \tilde{X}_r)\mathbf{1}_{[0,\tau_n]}dhd(\tilde{V}-\tilde{V}^c)_r.
\end{align*}
Then, its convergence can be proven by the same method as the third term. Meanwhile, the part $\sum_{s<r\leq t}\{ \partial_{\mu}F((r-)\wedge \tau_n,\mu^n_{r-},\tilde{X}_r^n)\Delta \tilde{X}^n_r\}$ can be absorbed into the third term to become as $\int_s^t\, \partial_\mu F((r-)\wedge \tau_n,\mu^n_{r-},\tilde{X}_r^n)d(\tilde{X}^n)^c_r$.\\
Combining the above convergence results, the proof of the desired formula is completed.

{\bf{Step 3}} Approximation by cylindrical functions.\\
Now we prove that formula (\ref{itowentzellc1}) holds for $RF-Partially-\mathcal{C}^2$ functions.
First, by Proposition \ref{cylindricalapproximationprop}, for any fixed $(\omega,t)\in\Omega\times[0,T]$, there exist two sequences of "cylindrical functions" $\{G^n\}$ and $\{H^n\}$, such that 
\begin{eqnarray*}&&
(G^n,\partial_{\mu}G^n,\frac{\delta G^n}{\delta \mu},H^n,\partial_{\mu}H^n,\frac{\delta H^n}{\delta \mu})
\to (G,\partial_{\mu}G,\frac{\delta G}{\delta \mu},H,\partial_{\mu}H,\frac{\delta H}{\delta \mu}),\quad \mbox{as}\ n\to\infty
\end{eqnarray*}
holds $(\omega,t,\mu)$-pointwisely.
However, note that here $n$ depends on $t$ and $\omega$. Hence for a fixed $n$, $G^n_{\cdot}(\omega,\mu)$ may not be measurable, which implies that $\int_{0}^{t}G_s^n(\mu)ds$ may not be well-defined. 
The same problem occurs with $\int_{0}^{t}H_s^n(\mu)dY_s$. So $F^n$ is not well-defined, which leads to the inapplicability of Theorem 2.
Therefore, we need some techniques to reform the approximation.
\begin{lemma}\label{approximationlemma}
There exists a series of adapted functions $\{F^n\}_{n\geq 1}$ satisfying 
\begin{eqnarray*}
F^n(t,\mu)=F^n_0(\mu)+\int_{0}^{t}\, G_r^n(\mu)dr+\int_{0}^{t} \, H_r^n(\mu)dY_r,
\end{eqnarray*}
where $F^n_0(\cdot)$ and $\ G_{r}^n(\omega,\cdot),H_{r}^n(\omega,\cdot),\forall (\omega,r)\in\Omega\times[0,t]$ are cylindrical functions, such that, for each fixed $(t,\mu,x)\in[0,T]\times\mathcal{P}_2(\mathbb{R}^d)\times\mathbb{R}^d$, the following convergences hold almost surely
\begin{eqnarray*}&&
\partial F^n(t,\mu)\to \partial F(t,\mu),\ \ \ \ \partial F^n_0(\mu)\to \partial F_0(\mu),\\&&
\int_{0}^{t}\, \partial G_r^n(\mu,x)dr \to \int_{0}^{t}\, \partial G_r(\mu,x)dr,\ \ \ \ \int_{0}^{t}\, \partial H_r^n(\mu,x)dY_r \to \int_{0}^{t}\, \partial H_r(\mu,x)dY_r,
\end{eqnarray*}
where $\partial \phi$ represents all the derivatives of $\phi$ including itself. 
Moreover, $\partial F_0$ is bounded by $\partial F$ almost surely, and $\partial G$, $\partial H$ are $L^2(\Omega\times[0,T])$-convergent, i.e. for any $(t,\mu,x)\in[0,T]\times\mathcal{P}_2(\mathbb{R}^d)\times\mathbb{R}^d$, it holds
\begin{align*}
\mathbb{E}\int_{0}^{t}\, |\partial K_r^n(\mu,x)-\partial K_r(\mu,x)|^2dr\to 0,\ \ \mbox{as}\ n\to\infty,\ \ {\text{with}}\ K:=G,H.
\end{align*}
\end{lemma}
From Step 1 we know that the It{\^o}-Wentzell-Lions formula holds for the approximating functions in Lemma \ref{approximationlemma}, 
\begin{align}\label{eq cylindrcal itowentzell}
    F^n(t,\mu_t)-&F^n(s,\mu_s)=\int_{s}^{t}  \,G^n_r(\mu_{r-})dr +\int_{s}^{t}  \,H^n_r(\mu_{r-})dY_r\nonumber\\&
    +\tilde{\mathbb{E}}\bigg[ \underbrace{ \int_{s}^{t}  \,\partial_{\mu}F^n(r-,\mu_{r-},\tilde{X}_{r-}) d\tilde{X}_r}_{J_1}+\underbrace{\frac{1}{2}\partial_x\partial_{\mu}F^n(r-,\mu_{r-},\tilde{X}_{r-}):d[\tilde{X},\tilde{X}]_r^c }_{J_2}\bigg]\nonumber\\&
    +\sum_{s<r\leq t}\left\{F^n(r,\mu_r)-F^n(r-,\mu_{r-})-H^n_r(\mu_{r-})\Delta Y_r\right\}\nonumber\\&
    +\tilde{\mathbb{E}}\bigg[ \underbrace{\sum_{s<r\leq t}\left\{ (\frac{\delta F^n}{\delta\mu}(r-,\mu_{r-},\tilde{X}_r)-\frac{\delta F^n}{\delta\mu}(r-,\mu_{r-},\tilde{X}_{r-}))\mathbf{1}_{\{\mu_r=\mu_{r-}\}}\right\}}_{J_3}\nonumber\\&\qquad
    -\sum_{s<r\leq t}\left\{ \partial_{\mu}F^n(r-,\mu_{r-},\tilde{X}_{r-})\Delta \tilde{X}_r\right\}\bigg],\quad a.s.
    \end{align} 
Thanks to Lemma \ref{approximationlemma}, \eqref{itowentzellc1} can be approximated via \eqref{eq cylindrcal itowentzell}, the convergence of all terms are clear except for $J_1,\ J_2$ and $J_3$. Next, let us explain the convergence of these three terms. 
\\First, we prove that $J_1$ is convergent in the following sense
\begin{align}\label{eq cylindrical approximation 1}
\tilde{\mathbb{E}}\bigg[ \int_s^t\, \partial_{\mu}F^n(r-,\mu_{r-},\tilde{X}_{r-})d\tilde{X}_r \bigg]\to\tilde{\mathbb{E}}\bigg[ \int_s^t\, \partial_{\mu}F(r-,\mu_{r-},\tilde{X}_{r-})d\tilde{X}_r \bigg],\ \ \ a.s.
\end{align}
It follows the expressions of $F^n$ and $F$ that
\begin{align}\label{eq cylindrical approximation decomposition}
&\mathbb{E}\bigg[\Big|\tilde{\mathbb{E}}\Big[ \int_s^t\, (\partial_{\mu}F^n-\partial_\mu F)d\tilde{X}_r \Big]\Big|^2\bigg]\nonumber\\
\leq&\, 3\mathbb{E}\tilde{\mathbb{E}}\bigg[ \Big|\int_s^t\, (\partial_{\mu}F_0^n-\partial_\mu F_0)d\tilde{X}_r \Big|^2+\Big|\int_s^t\, \int_0^r\,(\partial_{\mu}G_u^n-\partial_\mu G_u)dud\tilde{X}_r \Big|^2\nonumber\\
&\qquad+\Big|\int_s^t\, \int_0^r\,(\partial_{\mu}H_u^n-\partial_\mu H_u)dY_ud\tilde{X}_r \Big|^2\bigg],
\end{align}
here and after we omit the variables of the functions for convenience. For the first term on the right side of \eqref{eq cylindrical approximation decomposition}, by the definition of $X$, we have 
\begin{align*}
\mathbb{E}\tilde{\mathbb{E}}\bigg[ \Big|\int_s^t\, (\partial_{\mu}F_0^n-\partial_\mu F_0)d\tilde{X}_r \Big|^2\bigg]\leq 2\mathbb{E}\tilde{\mathbb{E}}\bigg[ \Big|\int_s^t(\partial_{\mu}F_0^n-\partial_\mu F_0)d\tilde{V}_r \Big|^2+\Big|\int_s^t(\partial_{\mu}F_0^n-\partial_\mu F_0)d\tilde{M}_r \Big|^2\bigg].
\end{align*}
Then, applying Cauchy-Schwarz inequality yields
\begin{align*}
\mathbb{E}\tilde{\mathbb{E}}\bigg[ \Big|\int_s^t\, (\partial_{\mu}F_0^n-\partial_\mu F_0)d\tilde{V}_r \Big|^2\bigg]\leq \mathbb{E}\tilde{\mathbb{E}}\bigg[ \int_s^t\, |\partial_{\mu}F_0^n-\partial_\mu F_0|^2d\tilde{V}_r\cdot Var(\tilde{V})_{[s,t]} \bigg],
\end{align*}
and It{\^o} isometry provides
\begin{align*}
\mathbb{E}\tilde{\mathbb{E}}\bigg[ \Big|\int_s^t\, (\partial_{\mu}F_0^n-\partial_\mu F_0)d\tilde{M}_r \Big|^2\bigg]=\mathbb{E}\tilde{\mathbb{E}}\bigg[ \int_s^t\, |\partial_{\mu}F_0^n-\partial_\mu F_0|^2d[\tilde{M},\tilde{M}]_r \bigg].
\end{align*}
Note that $|\partial_\mu F^n_0|\leq |\partial_\mu F_0|\leq C$, and by the dominated convergence theorem, we obtain $$\mathbb{E}\tilde{\mathbb{E}}\bigg[ \Big|\int_s^t\, (\partial_{\mu}F_0^n-\partial_\mu F_0)d\tilde{X}_r\Big|^2\bigg]\to 0, \quad \mbox{as}\ n\to\infty.$$\\
For the second term, still by the definition of $X$, we have
\begin{align*}
\mathbb{E}\tilde{\mathbb{E}}\bigg[\Big|\int_s^t\, \int_0^r\,(\partial_{\mu}G_u^n-\partial_\mu G_u)dud\tilde{X}_r \Big|^2\bigg]\leq &\,2\mathbb{E}\tilde{\mathbb{E}}\bigg[\Big|\int_s^t\, \int_0^r\,(\partial_{\mu}G_u^n-\partial_\mu G_u)dud\tilde{V}_r \Big|^2\\
&\qquad+\Big|\int_s^t\, \int_0^r\,(\partial_{\mu}G_u^n-\partial_\mu G_u)dud\tilde{M}_r \Big|^2\bigg].
\end{align*}
Then by Fubini's Theorem and Cauchy-Schwarz inequality, we obtain
\begin{align*}
\mathbb{E}\tilde{\mathbb{E}}\bigg[\Big|\int_s^t\int_0^r(\partial_{\mu}G_u^n-\partial_\mu G_u)dud\tilde{V}_r \Big|^2\bigg]\leq &\,\tilde{\mathbb{E}}\bigg[\int_s^t\mathbb{E}\Big|\int_0^r(\partial_{\mu}G_u^n-\partial_\mu G_u)du\Big|^2d\tilde{V}_r\cdot Var(\tilde{V})_{[s,t]}\bigg]\\
\leq &\,\tilde{\mathbb{E}}\bigg[\int_s^t\mathbb{E}\Big[\int_0^r|\partial_{\mu}G_u^n-\partial_\mu G_u|^2du\Big]rd\tilde{V}_r\cdot Var(\tilde{V})_{[s,t]}\bigg]
\end{align*}
and by It{\^o} isometry, we have
\begin{align*}
\mathbb{E}\tilde{\mathbb{E}}\bigg[\Big|\int_s^t\, \int_0^r\,(\partial_{\mu}G_u^n-\partial_\mu G_u)dud\tilde{M}_r \Big|^2\bigg]=&\tilde{\mathbb{E}}\bigg[\int_s^t\,\mathbb{E}\Big|\int_0^r\, (\partial_{\mu}G_u^n-\partial_\mu G_u)du\bigg|^2d[\tilde{M},\tilde{M}]_r\bigg]\\
\leq &\tilde{\mathbb{E}}\bigg[\int_s^t\,\mathbb{E}\Big[\int_0^r |\partial_{\mu}G_u^n-\partial_\mu G_u|^2du\Big]rd[\tilde{M},\tilde{M}]_r\bigg].
\end{align*}
Note that $\mathbb{E}[\int_0^t\,|\partial_\mu G_u(\mu,x)-\partial_\mu G^n_u(\mu,x)|^2du]\to 0, \forall x\in\mathbb{R}^d,\mu\in\mathcal{P}_2(\mathbb{R}^d)$, and by the dominated convergence theorem, we get
$$\mathbb{E}\tilde{\mathbb{E}}\bigg[\Big|\int_s^t\, \int_0^r\,(\partial_{\mu}G_u^n-\partial_\mu G_u)dud\tilde{X}_r \Big|^2\bigg]\rightarrow0,\quad \mbox{as}\ n\rightarrow\infty.$$ 
\\For the third term, similarly, we have
\begin{align*}
\mathbb{E}\tilde{\mathbb{E}}\bigg[\Big|\int_s^t\int_0^r(\partial_{\mu}H_u^n-\partial_\mu H_u)dY_ud\tilde{X}_r \Big|^2\bigg]\leq &\,2\mathbb{E}\tilde{\mathbb{E}}\bigg[\Big|\int_s^t\int_0^r(\partial_{\mu}H_u^n-\partial_\mu H_u)dY_ud\tilde{V}_r\Big|^2\\
&\qquad+\Big|\int_s^t\int_0^r(\partial_{\mu}H_u^n-\partial_\mu H_u)dY_ud\tilde{M}_r\Big|^2\bigg].
\end{align*}
Still by Fubini's Theorem and Cauchy-Schwarz inequality, it holds
\begin{align*}
\mathbb{E}\tilde{\mathbb{E}}\bigg[\Big|\int_s^t\int_0^r(\partial_{\mu}H_u^n-\partial_\mu H_u)dY_ud\tilde{V}_r \Big|^2\bigg]\leq \tilde{\mathbb{E}}\bigg[\int_s^t\,\mathbb{E}\Big|\int_0^r(\partial_{\mu}H_u^n-\partial_\mu H_u)dY_u\Big|^2d\tilde{V}_r\cdot Var(\tilde{V})_{[s,t]}\bigg].
\end{align*}
It follows Lemma 1 that, for any $x\in\mathbb{R}^d,\mu\in\mathcal{P}_2(\mathbb{R}^d)$,
\begin{align*}
\mathbb{E}\bigg[\Big|\int_0^r(\partial_{\mu}H_u^n(\mu,x)-\partial_\mu H_u(\mu,x))dY_u\Big|^2 \bigg]\to 0.
\end{align*}
Then, we obtain due to the dominated convergence theorem that
\begin{align*}
\mathbb{E}\tilde{\mathbb{E}}\bigg[\Big|\int_s^t\int_0^r(\partial_{\mu}H_u^n-\partial_\mu H_u)dY_ud\tilde{V}_r \Big|^2\bigg]\to 0.
\end{align*}
Again, by It{\^o} isometry and the dominated convergence theorem, we have
\begin{align*}
\mathbb{E}\tilde{\mathbb{E}}\bigg[\Big|\int_s^t\int_0^r(\partial_{\mu}H_u^n-\partial_\mu H_u)dY_ud\tilde{M}_r \Big|^2\bigg]=&\,\tilde{\mathbb{E}}\bigg[\int_s^t\mathbb{E}\Big|\int_0^r(\partial_{\mu}H_u^n-\partial_\mu H_u)dY_u\Big|^2d[\tilde{M},\tilde{M}]_r\bigg]\\
\rightarrow&\, 0, \ \ \ {\text {as}} \ \ n\to\infty.
\end{align*}
Combining the above convergence results with \eqref{eq cylindrical approximation decomposition}, we get \eqref{eq cylindrical approximation 1} by taking a subsequence. 
\\The convergence of $J_2$ can be obtained by using the same method as $J_1$.

Finally, we prove the convergence of $J_3$, i.e.
\begin{eqnarray*}&&
\tilde{\mathbb{E}}\sum_{s<r\leq t}\left\{ (\frac{\delta F^n}{\delta\mu}(r-,\mu_{r-},\tilde{X}_r)-\frac{\delta F^n}{\delta\mu}(r-,\mu_{r-},\tilde{X}_{r-}))\mathbf{1}_{\{\mu_r=\mu_{r-}\}}\right\}\\&&
\to \tilde{\mathbb{E}}\sum_{s<r\leq t}\left\{ (\frac{\delta F}{\delta\mu}(r-,\mu_{r-},\tilde{X}_r)-\frac{\delta F}{\delta\mu}(r-,\mu_{r-},\tilde{X}_{r-}))\mathbf{1}_{\{\mu_r=\mu_{r-}\}}\right\},\ \ \ a.s.
\end{eqnarray*}
We know from Remark 3.2 in \cite{guo2023ito} that
\begin{align*}
&\sum_{s<r\leq t}\Big(\frac{\delta F^n}{\delta\mu}(r-,\mu_{r-},\tilde{X}_r)-\frac{\delta F^n}{\delta\mu}(r-,\mu_{r-},\tilde{X}_{r-})\Big)\mathbf{1}_{\{\mu_r=\mu_{r-}\}}\\
=&\sum_{s<r\leq t}\int_0^1\, \partial_\mu F^n(r-,\mu_{r-},X_{r-}+h\Delta X_r)\cdot \Delta Xdh\\
=&\int_s^t\,\int_0^1\, \partial_\mu F^n(r-,\mu_{r-},X_{r-}+h\Delta X_r)dhd(V-V^c)_r.
\end{align*}
Then by the same method as $J_1$, the desired result can be deduced.$\hspace{3.8cm}\Box$\\

Before proving Lemma \ref{approximationlemma}, we introduce the following space.
\begin{definition}
A function $g:[0,T]\times\mathcal{P}_2(\mathbb{R}^d)\times(\mathbb{R}^d)^{2n}\to\mathbb{R}$ is said to belong to $\mathcal{C}_0^{0,1,\infty}([0,T]\times\mathcal{P}_2(\mathbb{R}^d)\times(\mathbb{R}^d)^{2n})$ if\\
(i) $g(\cdot,\mu,x)$ is continuous w.r.t. $t$ for any $(\mu,x)$;\\
(ii) $g(t,\mu,\cdot)$ is $C_0^{\infty}$ w.r.t $x$ for any $(t,\mu)$;\\
(iii) $g(t,\cdot,x)$ is $\mathcal{C}^1$-differentiable w.r.t $\mu$ for any $(t,x)$.
\end{definition}

\vspace{3mm}
{\textit{Proof of Lemma \ref{approximationlemma}.}}
First, since $F_0(\mu)$ is deterministic, it is clear that there exists a series of $\mathcal{C}^2$ cylindrical functions $\{F^n_0(\mu) \}_{n=1}^{\infty}$ such that $\partial F^n_0(\mu)\to \partial F_0(\mu)$. Note that 
\begin{align*}
F^n_0(\mu)=F_0(T^*_n\mu),\ \ \partial F_0^n(\mu,x)=T_n(\partial F_0(T^*_n\mu,\cdot))(x),
\end{align*}
where $T_n$ is defined in the proof of Proposition 3. Therefore,
\begin{align*}
|F^n_0(\mu)|\leq |F_0(\mu)|,\ \ |\partial_\mu F_0^n(\mu,x)|\leq \sup_{y\in\mathbb{R}^d}|\partial F_0(\mu,y)|.
\end{align*}
Then, to prove $\int_{0}^{t}\, \partial G_r^n(\mu)dr \to \int_{0}^{t}\, \partial G_r(\mu)dr$, we take the approximation in following manners:\\
First, we assume that the function $G_\cdot$ is continuous w.r.t time.
\begin{enumerate}
    \item Take $\{t_i \}_{i=1}^{\infty}$ a dense subset of $[0,T]$, by Lemma \ref{appendixlemma2} in Appendix A, there exists a series of deterministic functions $g_j\in \mathcal{C}_0^{0,1,\infty}([0,T]\times\mathcal{P}_2(\mathbb{R}^d)\times(\mathbb{R}^d)^{2j}),j=1,2,\ldots$ such that, for any $\mu$,
    \begin{align}\label{L2convengence}
    \lim_{j\to\infty}\mathbb{E}\bigg[ \int_0^T\, |G_s(\mu)-g_j(s,\mu,U_{s\wedge t_1},\ldots,U_{s\wedge t_j},N_{s\wedge t_1},\ldots,N_{s\wedge t_j})|^2ds \bigg]=0.
    \end{align}
    \item Let $\Pi^n:=\{0=t^n_0<t^n_1<\cdots<t^n_n=T\}$ be a partition of $[0,T]$ such that $\max_{k}\left\lvert t_{k+1}^n-t^n_k\right\rvert\rightarrow0$ and $\max_k \sup_{r\in (t_k^n,t_{k+1}^n]}\left\lvert Y_{r-}-Y_{t^n_{k}}\right\rvert\rightarrow0$ as $n\rightarrow\infty$. 
    For any given $j$, we consider a time-simple function $\sum_{k=0}^{n-1} g_j(t_k^n,\mu,x_1,\ldots,x_{2j})\mathbf{1}_{[t^n_k,t^n_{k+1})}$.
    \item For any fixed $(t^n_k,\mu)$, $g_j(t^n_k,\mu,\cdot)\in C_0^{\infty}((\mathbb{R}^d)^{2j})$, hence it can be approximated by the space-simple functions $\sum_{l=1}^{m} g_j(t^n_k,\mu,\mathbf{x}_{A^m_{l}})\mathbf{1}_{A^m_l}$, where $\{A^m_l\}_{l=1}^{m}$ is a partition of $(\mathbb{R}^d)^{2j}$ 
    such that $\max_{l}|A^m_l|\to 0$ as $m\to\infty$, and $\mathbf{x}_{A^m_{l}}\in A^m_{l}$. That is
    \begin{align*}
    \lim_{m\to\infty}\sum_{l=1}^{m} g_j(t^n_k,\mu,\mathbf{x}_{A^m_{l}})\mathbf{1}_{A^m_l}(\mathbf{x})=g_j(t^n_k,\mu,\mathbf{x}).
    \end{align*}
    Then, by the definition of limit, we can take $m_n$ for each $n$ such that
    \begin{align*}
    \sup_{\mathbf{x}\in(\mathbb{R}^d)^{2n}}|\sum_{l=1}^{m_n}g_j(t^n_k,\mu,\mathbf{x}_{A^{m_n}_l})\mathbf{1}_{A^{m_n}_l}(\mathbf{x})-g_j(t^n_k,\mu,\mathbf{x})|<\frac{1}{n}.
    \end{align*}
    \item Next by Proposition \ref{cylindricalapproximationprop}, there exists a cylindrical function $g_j^{(n)}(t^n_k,\mu,\mathbf{x}_{A^{m_n}_{l}})$ satisfying
    \begin{align*}
    |g_j^{(n)}(t^n_k,\mu,\mathbf{x}_{A^{m_n}_{l}})-g_j(t^n_k,\mu,\mathbf{x}_{A^{m_n}_{l}})|<\frac{1}{n}.
    \end{align*}
    Set $g^{(n)}_j(t^n_k,\mu,\mathbf{x} ):=\sum_{l=1}^{\infty}g_j^{(n)}(t^n_k,\mu,\mathbf{x}_{A^{m_n}_{l}})\mathbf{1}_{A^{m_n}_l}(\mathbf{x})$. It is clear that $g^{(n)}_j(t^n_k,\mu,\cdot)$ is a measurable function for any $\mu$.
    \item Finally, for any $(\omega,s)\in\Omega\times[0,T]$, the approximating functions of $G$ are defined as follows
    \begin{align}\label{eq approximate G finally}
    G^{j,n}_s(\omega,\mu):=\sum_{k=0}^{n-1}g^{(n)}_j(t^n_k,\mu,U_{s\wedge t_1},\ldots,U_{s\wedge t_j},N_{s\wedge t_1},\ldots,N_{s\wedge t_j})\mathbf{1}_{[t^n_k,t^n_{k+1})}(s).\end{align}
    From the above argument we know that $\{G^{j,n}\}$ is a series of $\mathbb{F}$-adapted functions. Moreover, $\{G^{j,n}_s(\omega,\cdot)\}$ is  a series of $\mathcal{C}^1(\mathcal{P}_2(\mathbb{R}^d))$-cylindrical functions given any $(\omega,s)\in\Omega\times [0,T]$.
\end{enumerate}
Now we verify the convergence of $\lim_{j}\lim_{n}\int_0^t\,\partial G_r^{j,n}(\mu)dr$. We only show the convergence of $G$ itself and those of its derivatives follow the same argument. A direct calculus yields
\begin{align}\label{approximatedecomposition}
&\mathbb{E}\bigg[ \int_0^t\, |G^{j,n}_r(\mu)-G_r(\mu)|^2dr \bigg]\nonumber\\
\leq &\,C \mathbb{E}\bigg[ \int_0^t\,|G^{j,n}_r(\mu)-\sum_{k=0}^{n-1}G_{t^n_{k}}(\mu)\mathbf{1}_{[t^n_k,t^n_{k+1})}(r)|^2dr \bigg]\nonumber\\&+C\mathbb{E}\bigg[ \int_0^t\,|\sum_{k=0}^{n-1}G_{t^n_{k}}(\mu)\mathbf{1}_{[t^n_k,t^n_{k+1})}(r)-G_r(\mu)|^2dr \bigg]\nonumber\\
\leq &\,C \mathbb{E}\bigg[ \int_0^t\,|\sum_{k=0}^{n-1}g^{(n)}_j(t^n_k,\mu,\mathbf{UN}_r^j)\mathbf{1}_{[t^n_k,t^n_{k+1})}(r)-\sum_{k=0}^{n-1}g_j(t^n_k,\mu,\mathbf{UN}_r^j)\mathbf{1}_{[t^n_k,t^n_{k+1})}(r)|^2dr \bigg]\nonumber\\
&+C\mathbb{E}\bigg[ \int_0^t\, |\sum_{k=0}^{n-1}g_j(t^n_k,\mu,\mathbf{UN}_r^j)\mathbf{1}_{[t^n_k,t^n_{k+1})}(r)-\sum_{k=0}^{n-1}G_{t^n_{k}}(\mu)\mathbf{1}_{[t^n_k,t^n_{k+1})}(r)|^2dr \bigg]\nonumber\\
&+C\mathbb{E}\bigg[ \int_0^t\,|\sum_{k=0}^{n-1}G_{t^n_{k}}(\mu)\mathbf{1}_{[t^n_k,t^n_{k+1})}(r)-G_r(\mu)|^2dr \bigg],
\end{align}
where $\mathbf{UN}_r^j:=(U_{r\wedge t_1},\ldots,U_{r\wedge t_j},N_{r\wedge t_1},\ldots,N_{r\wedge t_j})$.\\
Then for the first term on the right hand side of \eqref{approximatedecomposition}, it holds
\begin{align*}
&\mathbb{E}\bigg[ \int_0^t\,|\sum_{k=0}^{n-1}g^{(n)}_j(t^n_k,\mu,\mathbf{UN}_r^j)\mathbf{1}_{[t^n_k,t^n_{k+1})}(r)-\sum_{k=0}^{n-1}g_j(t^n_k,\mu,\mathbf{UN}_r^j)\mathbf{1}_{[t^n_k,t^n_{k+1})}(r)|^2dr \bigg]\\
\leq &\, n\mathbb{E}\bigg[ \sum_{k=0}^{n-1}\int_{t^n_k}^{t^n_{k+1}}|g^{(n)}_j(t^n_k,\mu,\mathbf{UN}_r^j)-g_j(t^n_k,\mu,\mathbf{UN}_r^j)|^2dr \bigg]<n\frac{4t}{n^2}\rightarrow0,\ \ \mbox{as}\ n\to\infty.
\end{align*} 
For the second term, by the time-continuity of $G$ and $g_j$, we have
\begin{align*}
&\mathbb{E}\bigg[ \int_0^t\, |\sum_{k=0}^{n-1}g_j(t^n_k,\mu,\mathbf{UN}_r^j)\mathbf{1}_{[t^n_k,t^n_{k+1})}(r)-\sum_{k=0}^{n-1}G_{t^n_{k}}(\mu)\mathbf{1}_{[t^n_k,t^n_{k+1})}(r)|^2dr \bigg]\\
\to &\, \mathbb{E}\bigg[ \int_0^t\,|g_j(r,\mu,\mathbf{UN}^j_r)-G_r(\mu)|^2dr \bigg], \  \mbox{ as } n\rightarrow\infty.
\end{align*}
Furthermore, due to \eqref{L2convengence}, the above term converges to $0$ as $j\to\infty$.\\
For the last term, still by the time-continuity of $G$, one gets 
\begin{align*}
\mathbb{E}\bigg[ \int_0^t\,|\sum_{k=0}^{n-1}G_{t^n_{k}}(\mu)\mathbf{1}_{(t^n_k,t^n_{k+1}]}(r)-G_r(\mu)|^2dr \bigg]\to 0,\  \mbox{ as } n\to\infty.
\end{align*}

Hence, until now, we prove the convergence of $\int_0^tG^{j,n}_r(\mu)dr$ in $L^2(\Omega)$. Then we can take a subsequence and still denote it as $\{G^{j,n} \}_{j,n=1}^{\infty}$ such that 
\begin{align*}
\int_0^t\, G^{j,n}_r(\mu)dr\to \int_0^t\, G_r(\mu)dr,\ \ \ a.s.
\end{align*}
If $G$ is only time-measurable, we need to seek for a series of time-continuous and $\mathbb{F}$-adapted processes $\{I_m(\mu)\}$ such that 
\begin{align*}
\int_0^t\, I_m(r,\mu)dr\to \int_0^t\, G_r(\mu)dr,\ \ \ a.s.
\end{align*}
Define
\begin{align}\label{In}
I_m(t,\mu):=\int_0^t\psi_m(s-t)G_s(\mu)ds,
\end{align}
where $\psi_m$ is a non-negative continuous function on $\mathbb{R}$ such that 
\begin{align*}
\psi_m(x)=0, \ {\text{for}}\ x\leq-\frac{1}{m} \ {\text{and}}\ x\geq0;\ \ \int_{-\infty}^{\infty}\psi_m(x)dx=1.
\end{align*}
Since $G$ is $\mathbb{F}$-adapted, it is clear that $I_m(t,\mu)$ is $\mathscr{F}_t$-measurable.
Due to the boundedness of $G$, it is easy to check that $\int_0^tI_m(s,\mu)ds\to\int_0^tG_s(\mu)ds$ in $L^2$. Then, we can take a subsequence which is still denoted as $\{\int_0^tI_m(s,\mu)ds\}_m$ such that $\int_0^tI_m(s,\mu)ds\to\int_0^tG_s(\mu)ds$, a.s.
It is known from \eqref{In} that $I_m$ has the same differentiability w.r.t. $\mu$ as $G$. Hence we can construct the cylindrical functions $\{I_m^{j,n}\}$ to approximate $I_m$ by using the same method as in deriving \eqref{eq approximate G finally}. Then, let $m,n,j\to\infty$, it is easy to see that we have
\begin{align*}
\lim_{m,n,j\to\infty}\mathbb{E}\bigg[ \int_0^t\Big|I^{j,n}_m(r,\mu)-G_r(\mu)\Big|^2dr \bigg]= 0.
\end{align*}
To prove $\int_{0}^{t}\partial H_r^n(\mu)dY_r \to \int_{0}^{t}\partial H_r(\mu)dY_r$, when $H$ is continuous w.r.t. time, we use the same argument as before to get 
\begin{align*}
\mathbb{E}&\bigg[ \int_0^t\Big|H^{j,n}_r(\mu)-H_r(\mu)\Big|^2dr \bigg]\to 0,\ \ \ j,n\to \infty.
\end{align*}
Thus by It{\^o} isometry, we obtain
\begin{align*}
\lim_{j,n\to\infty}\mathbb{E}\bigg[ \bigg| \int_0^t\,H^{j,n}_r(\mu)dN_r- \int_0^t\,H_r(\mu)dN_r\bigg|^2 \bigg]= 0.
\end{align*}
Then, we can take a subsequence and still denote as $\{H^{j,n}\}_{j,n=1}^{\infty}$ such that \
\begin{align*}
\lim_{j,n\to\infty}\int_0^t\,H^{j,n}_r(\mu)dN_r=\int_0^t\,H_r(\mu)dN_r,\ \ \lim_{j,n\to\infty}\int_0^t\,H^{j,n}_r(\mu)dU_r=\int_0^t\,H_r(\mu)dU_r,\ a.s.
\end{align*}
If $H$ is only time-measurable, we extend the methodology used for $G$ to derive the convergence results for $H$. Moreover, by the definition of $H^{j,n}$, it is clear that its derivatives are $\mathbb{F}-$predictable.\\
Therefore, the lemma is proved.  $\hspace{9.5cm}\Box$

\section{It{\^o}-Wentzell-Lions formula for flows of conditional measures}{\label{c2}}

Motivated by the mean-field game problems with common noise, we shall establish the It{\^o}-Wentzell-Lions formula for flows of conditional measures. Let $X$ and $Y$ be respectively $\mathbb{R}^d-$ and $\mathbb{R}^l$-valued semimartingales on the complete filtered probability space $(\Omega,\mathscr{F},\mathbb{F}=(\mathscr{F}_t)_{t\in [0,T]},\mathbb{P})$, where $\mathscr{F}\supset \mathscr{F}_T$ is rich enough. Suppose $\mathbb{G}=(\mathscr{G}_t)_{t\in[0,T]}$ is a sub-filtration of $\mathbb{F}$. In this section, we set for any $t\in[0,T]$,  $\mu_t:=\mathbb{P}_{X_t|\mathscr{G}_t}$ the conditional probability measure given $\mathscr{G}_t$ and develop the It{\^o}-Wentzell-Lions formula for $F:\Omega\times[0,T]\times\mathcal{P}_2(\mathbb{R}^d)\rightarrow\mathbb{R}$ defined in \eqref{decompositionFmu} with stricter regularity conditions than those in Section 3.
\begin{definition}
$F$ is said to be $RF-Fully-\mathcal{C}^2$ if $F$ is $RF-Partially-\mathcal{C}^2$, and moreover, it is twice differentiable w.r.t $\mu$ with a continuous mapping $(t,\mu,x_1,x_2)\to \partial_{\mu\mu}F(t,\mu,x_1,x_2)$ satisfying 
\begin{eqnarray*}&&
\sup_{\mu\in \mathcal{P}_2(\mathbb{R}^d),x_1,x_2\in \mathbb{R}^d}\big\lvert\partial_{\mu\mu}F_0(\mu,x_1,x_2)\big\rvert  \leq C,\\&&
 \int_0^T\,\sup_{\mu\in \mathcal{P}_2(\mathbb{R}^d),x_1,x_2\in \mathbb{R}^d}\big\lvert\partial_{\mu\mu}K(t,\mu,x_1,x_2)\big\rvert dt \leq C,\quad a.s.,
\end{eqnarray*}
with $K:=G,H$.
\end{definition}
\begin{assumption}\label{filtrationassumption}
The sub-filtration $\mathbb{G}\subset \mathbb{F}$ satisfies the compatibility assumption (or the conditional independence condition). That is, for any $t\in[0,T]$, 
$\mathscr{F}_t$ and $\mathscr{G}_T$ are independent given $\mathscr{G}_t$, denoted as $\mathscr{F}_t\perp \mathscr{G}_T|\mathscr{G}_t$.
\end{assumption}
It is easy to see that Assumption \ref{filtrationassumption} provides that, for $\mathbb{F}$-adapted process $X$, $\mathbb{P}_{X_t|\mathscr{G}_t}=\mathbb{P}_{X_t|\mathscr{G}_T}\ a.s.$  $\forall t\in[0,T]$. So we denote $\mu_t:=\mathbb{P}_{X_t|\mathscr{G}_t}=\mathbb{P}_{X_t|\mathscr{G}_T}$ throughout this section.
\\Next, we recall the definition and some useful propositions of conditionally independent copies. For the details, the readers are refereed to \cite{guo2024ito} (Section 3.3). For a given probability space $(\Omega,\mathscr{F},\mathbb{P})$, we define $\overline{\Omega}:=\Omega^{n+1}$ as 
\begin{eqnarray*}
\overline{\Omega}:=\Omega^{n+1}:=\{ (\omega_0,\omega_1,\cdots,\omega_n)|\omega_i\in\Omega,i=0,\cdots,n \},
\end{eqnarray*}
with extended $\sigma$-algebra 
\begin{eqnarray*}
\overline{\mathscr{F}}:=\sigma\{ A_0\times A_1 \times \cdots\times A_n|A_i\in \mathscr{F},i=0,\cdots,n \} .
\end{eqnarray*}
To ensure that the copies are conditionally independent given a sub-$\sigma$-algebra $\mathscr{G}\subset \mathscr{F}$, we define a probability measure on $\overline{\Omega}$ as
\begin{eqnarray*}
\overline{\mathbb{P}}(A_0\times A_1 \times \cdots\times A_n):= \mathbb{E}\bigg[ \mathbf{1}_{A_0}\prod_{i=1}^n\mathbb{P}(A_i|\mathscr{G})  \bigg],\quad A_i\in \mathscr{F}.
\end{eqnarray*}
It is clear that $\overline{\mathbb{P}}$ is well-defined.
Moreover, define $\overline{\mathbb{F}}:=(\overline{\mathscr{F}_t})_{t\in[0,T]}$ by
\begin{align*}
\overline{\mathscr{F}}_t:=\sigma\{ A_0\times A_1 \times \cdots\times A_n|A_i\in \mathscr{F}_t,i=0,\cdots,n \} .
\end{align*}
The following proposition holds for any $n\in\mathbb{N}_+$, but we only use it for $n=2$.
\begin{proposition}\label{conditionalcopyprop}(\cite{guo2024ito}, Corollary 3.2)
Given $\mathbb{F}-$adapted semimartingales $X$ and $Y$ on a filtered probability space $(\Omega,\mathscr{F},\mathbb{F}=(\mathscr{F}_t)_{t\in [0,T]},\mathbb{P})$, and a sub $\sigma$-algebra $\mathscr{G}\subset\mathscr{F}$, 
there exists a unique enlarged probability space $(\overline{\Omega},\overline{\mathscr{F}},\overline{\mathbb{F}}=(\overline{\mathscr{F}}_t)_{t\in [0,T]},\overline{\mathbb{P}})$ defined as the above.  Define two conditional copies 
$(X',Y')$ and $(X'',Y'')$ of $(X,Y)$ by
\begin{eqnarray*}
(X',Y')(\omega_0,\omega_1,\omega_2)=(X,Y)(\omega_1),\quad (X'',Y'')(\omega_0,\omega_1,\omega_2)=(X,Y)(\omega_2),
\end{eqnarray*}
then, for any $t\in[0,T]$, it holds almost surely that
\begin{align}\label{conditionalcopy}
\mathbb{P}_{X_t,Y_t|\mathscr{G}}=\mathbb{P}_{X_t',Y_t'|\mathscr{G}}=\mathbb{P}_{X_t',Y_t'|\mathscr{F}}=\mathbb{P}_{X_t'',Y_t''|\mathscr{G}}=\mathbb{P}_{X_t'',Y_t''|\mathscr{F}}.
\end{align}
Moreover, $(X,Y),\ (X',Y')$ and $(X'',Y'')$ are conditionally independent given $\mathscr{G}$, and are also conditionally independent given $\mathscr{F}$.\\
Furthermore, if the sub-$\sigma$-algebra $\mathscr{G}=\mathscr{G}_T$ from a sub filtration $\mathbb{G}=(\mathscr{G}_t)_{t\in[0,T]}\subset\mathbb{F}$ which satisfies Assumption 3, the proposition still holds.
\end{proposition}
The following lemma is useful for simplifying the subsequent calculation.
\begin{lemma}\label{conditionalcopylemma}(\cite{guo2024ito}, Theorem 3.3)
Given semimartingales $X$, $Y$ on $(\Omega,\mathscr{F},\mathbb{F}=(\mathscr{F}_t)_{t\in [0,T]},\mathbb{P})$, and a sub-$\sigma$-algebra $\mathscr{G}\subset\mathscr{F}$, 
let $(X',Y')$ and $(X'',Y'')$ be the conditionally independent copies defined in Proposition \ref{conditionalcopyprop}. Then\\
(1) $X,\ Y,\ X',\ Y',\ X'',\ Y''$ are $\overline{\mathbb{F}} $-semimartingales;\\
(2) Suppose $\mathbb{E}[\left\lvert X_t\right\rvert]<\infty$ and $\mathbb{E}[\left\lvert Y_t\right\rvert ]<\infty$ for all $t\in [0,T]$, and denote
\begin{eqnarray*}&&
X_t^{\mathscr{G}}:=\overline{\mathbb{E}}[X_t|\mathscr{G}_T] =\overline{\mathbb{E}}[X_t|\mathscr{G}_t].
\end{eqnarray*}
Then $X^{\mathscr{G}}$ is a semimartingale w.r.t. $\mathbb{F},\ \overline{\mathbb{F}} $ and $\mathbb{G}$;\\
(3) If $X$ satisfies Assumption 1, so does $X^{\mathscr{G}}$ (changing $\mathbb{E}$ to $\overline{\mathbb{E}} $);\\
(4) Suppose $X,Y$ satisfy Assumption 1. Let $(Z_t)_{t\in[0,T]}$ be an $\mathbb{F}$-adapted c{\`a}dl{\`a}g process on $(\Omega,\mathscr{F},\mathbb{F}=(\mathscr{F}_t)_{t\in [0,T]},\mathbb{P})$, 
and $\left\lVert X\right\rVert_{\mathcal{H}_p}+ \left\lVert Z\right\rVert_{\mathcal{H}_q}<\infty$ with $\frac{1}{p}+\frac{1}{q}=1$, then
\begin{eqnarray}\label{eq conditionalexpectationchange}
\int_{0}^{t}  \,Z_{s-}dX_s^{\mathscr{G}}=\overline{\mathbb{E}}\left[ \int_{0}^{t}  \,Z_{s-}dX_s'\Big|\mathscr{F} \right]  .
\end{eqnarray}
Here $\left\lVert X\right\rVert_{\mathcal{H}_p}:=\mathbb{E}\left[\big(\left\lvert X_0\right\rvert+\left\lvert [M,M]_T\right\rvert^{\frac{1}{2}} +Var(V)_T \big)^p \right]^{\frac{1}{p}}$. 
Moreover, if $\left\lVert X\right\rVert_{\mathcal{H}_p}+\left\lVert Y\right\rVert_{\mathcal{H}_q}+\left\lVert Z\right\rVert_{\mathcal{H}_r}<\infty$ with 
$\frac{1}{p}+\frac{1}{q}+\frac{1}{r}=1$, then
\begin{eqnarray*}&&
\int_{0}^{t}  \,Z_{s-}d[X_s^{\mathscr{G}},Y_s^{\mathscr{G}}]=\overline{\mathbb{E}}\left[ \int_{0}^{t}  \,Z_{s-}d[X',Y'']_s\Big|\mathscr{F} \right],\\&&
\int_{0}^{t}  \,Z_{s-}d[X_s^{\mathscr{G}},Y_s]=\overline{\mathbb{E}}\left[ \int_{0}^{t}  \,Z_{s-}d[X',Y]_s\Big|\mathscr{F} \right].
\end{eqnarray*}
\end{lemma}
To satisfy the condition of Lemma \ref{conditionalcopylemma} (4), we need further assumption.
\begin{assumption}\label{assumption h bounded}
For a RF-Fully-$\mathcal{C}^2$ function $F$ with decomposition \eqref{decompositionFmu}, the $\mu$-derivative of $H$ is bounded, i.e. there exists a constant $C>0$ such that for all $t\in[0,T]$, $\mu\in\mathcal{P}_2(\mathbb{R}^d)$ and $x\in\mathbb{R}^d$, it holds 
\begin{align*}
|\partial_{\mu}H_t(\mu,x)|\leq C,\quad a.s.
\end{align*}
\end{assumption}
Now we state and prove our second main theorem.
\begin{theorem}\label{itowentzellc2thm}
Given semimartingales $X$, $Y$ satisfying Assumption 1, denote $\mu_t:=\mathbb{P}_{X_t|\mathscr{G}_t}$, with sub-filtration $\mathbb{G}=(\mathscr{G}_t)_{t\in [0,T]}\subset\mathbb{F}$ satisfying Assumption \ref{filtrationassumption}. 
Then, for any $RF-Fully-\cC^2$ function $F$ satisfying Assumptions 2 and 4, it holds almost surely, for $0\leq s\leq t\leq T$, 
\begin{align}\label{itowentzellc2}
&F(t,\mu_t)-F(s,\mu_s)=\int_{s}^{t}  \,G_r(\mu_{r-})dr+\int_{s}^{t}  \,H_r(\mu_{r-})dY_r\nonumber\\
&\quad+\sum_{s<r\leq t}\left\{ F(r,\mu_r)-F(r-,\mu_{r-})- H_r(\mu_{r-})\Delta Y_r\right\} \nonumber\\
&\quad+\overline{\mathbb{E}}\bigg[ \int_{s}^{t}  \,\partial_{\mu}F(r-,\mu_{r-},X'_{r-})dX'_r+\frac{1}{2}\partial_{\mu}\partial_x F(r-,\mu_{r-},X'_{r-}):d[X',X']^c_r \nonumber\\
&\qquad\quad+\frac{1}{2}\int_{s}^{t}  \,\partial_{\mu\mu}F(r-,\mu_{r-},X'_{r-},X''_{r-}):d[X',X'']^c_r+\int_{s}^{t}  \,\partial_{\mu}H_r(\mu_{r-},X'_{r-}):d[X',Y]^c_r\nonumber\\
&\qquad\quad+\sum_{s<r\leq t}\bigg( \frac{1}{2} \bigg( \frac{\delta^2F}{\delta\mu^2}(r-,\mu_{r-},X'_r,X''_r)-\frac{\delta^2F}{\delta\mu^2}(r-,\mu_{r-},X'_{r-},X''_r)  \nonumber\\
&\qquad\qquad\qquad\quad-\frac{\delta^2F}{\delta\mu^2}(r-,\mu_{r-},X'_r,X''_{r-})+\frac{\delta^2F}{\delta\mu^2}(r-,\mu_{r-},X'_{r-},X''_{r-})\bigg) \nonumber\\
&\qquad\quad+(\frac{\delta H_r}{\delta\mu}(\mu_{r-},X'_r)-\frac{\delta H_r}{\delta\mu}(\mu_{r-},X'_{r-}))\Delta Y_r\nonumber\\
&\qquad\quad+\frac{\delta F}{\delta\mu}(r-,\mu_{r-},X'_r)-\frac{\delta F}{\delta\mu}(r-,\mu_{r-},X'_{r-})\bigg)\mathbf{1}_{\{\mu_r=\mu_{r-}\}}-\partial_{\mu}F(r-,\mu_{r-},X'_{r-})\Delta X'_r  \Big|\mathscr{F}\bigg].
\end{align}
\end{theorem}
\begin{proof} The proof is divided into 3 steps.\\
{\bf{Step 1}} Construct the It{\^o}-Wentzell-Lions formula for cylindrical functions.\\
We use the same notations as in the proof of Theorem \ref{itowentzellc1thm}, except $Z^j_t:=\langle \eta_j,\mu_t\rangle=\mathbb{E}[\eta_j(X_t)|\mathscr{G}_T]$, $\forall t\in[0,T]$. 
Apply the standard It{\^o} formula to $\eta_j(X_t)$ and then take conditional expectation on both sides, we have
\begin{align}\label{itoforeta2}
Z^j_t-Z^j_s&=\mathbb{E}\bigg[ \int_{s}^{t}  \,\partial_x\eta_j(X_{r-})dX_r+\frac{1}{2}\int_{s}^{t}  \,\partial_{xx}\eta_j(X_{r-}):d[X,X]^c_r \nonumber\\
&\qquad\quad+\sum_{s<r\leq t}\{\eta_j(X_r)-\eta_j(X_{r-})-\partial_x\eta_j(X_{r-})\Delta X_r\} \Big|\mathscr{G}_T \bigg],\ \ \ a.s.
\end{align}
From Lemma \ref{conditionalcopylemma} (1), we know that $\{\mathbf{Z}_t:=(Z^1_t,\cdots,Z^N_t)\}_{t\in [0,T]}$ is a $\overline{\mathbb{F}}$-semimartingale. Assume that $X$ is bounded. Then, thanks to Theorem \ref{standitowentzellsemimart}, we obtain
\begin{equation}\label{itoforf2}\begin{split}
&F(t,\mu_t)-F(s,\mu_s)=f(t,\mathbf{Z}_t)-f(s,\mathbf{Z}_s)\\
=&\int_{s}^{t}  \,g_r(\mathbf{Z}_{r-})dr+\int_{s}^{t}  \,h_r(\mathbf{Z_{r-}})dY_r 
+\underbrace{\sum_j\int_{s}^{t}  \,\partial_{z_j}f(r-,\mathbf{Z}_{r-})dZ^j_r}_{J_1}\\
&+\underbrace{\frac{1}{2}\sum_{j,k}\int_{s}^{t}  \,\partial_{z_j z_k}f(r-,\mathbf{Z}_{r-})d[Z^j,Z^k]^c_r}_{J_2}+\underbrace{\sum_j \int_{s}^{t}  \,\partial_{z_j}h_r(\mathbf{Z}_{r-})d[Z^j,Y]^c_r}_{J_3}\\
&+\underbrace{\sum_{s<r\leq t}\{f(r,\mathbf{Z}_r)-f(r-,\mathbf{Z}_{r-})-\sum_j \partial_{z_j}f(r-,\mathbf{Z}_{r-})\Delta Z_r^j-h_r(\mathbf{Z}_{r-})\Delta Y_r\}}_{J_4}, \ \ a.s.
\end{split}\end{equation}
Use the same argument as the proof of Theorem \ref{itowentzellc1thm}, and note that the conditions of Lemma \ref{conditionalcopylemma} are satisfied,
we can get 
\begin{align*}
J_1=&\,\overline{\mathbb{E}}\bigg[ \int_{s}^{t}  \,\sum_{j}\partial_{z_j}f(r-,\mathbf{Z}_{r-})(\partial_x\eta_j(X'_{r-})dX'_r+\frac{1}{2}\partial_{xx}\eta_j(X_{r-}'):d[X',X']^c_r) \\
&\qquad+\sum_{s<r\leq t}\sum_{j}\partial_{z_j}f(r-,\mathbf{Z}_{r-})\{\eta_j(X_r)-\eta_j(X_{r-})-\partial_x\eta_j(X_{r-}')\Delta X_r\} \Big|\mathscr{F}\bigg] \\
=&\,\overline{\mathbb{E}}\bigg[ \int_{s}^{t}  \,\partial_{\mu}F(r-,\mu_{r-},X'_{r-})dX'_r+\frac{1}{2}\partial_{\mu}\partial_x F(r-,\mu_{r-},X'_{r-}):d[X',X']^c_r \\
&\qquad+\sum_{s<r\leq t}\left\{ \frac{\delta F}{\delta\mu}(r-,\mu_{r-},X'_r)-\frac{\delta F}{\delta\mu}(r-,\mu_{r-},X'_{r-})-\partial_{\mu}F(r-,\mu_{r-},X'_{r-})\Delta X'_r \right\} \Big|\mathscr{F}\bigg].
\end{align*}
For $J_2$, a direct calculation provides
\begin{align*}
&[\eta_j(X'),\eta_k(X'')]_t=[\eta_j(X'),\eta_k(X'')]^c_t+\sum_{0<s\leq t}\Delta \eta_j(X_s')\Delta \eta_k(X_s'')\\
=&\int_{0}^{t}\partial_x\eta_j(X'_{s-})\partial_x\eta_j(X'_{s-})^{\intercal  }:d[X',X'']^c_s+\sum_{0<s\leq t}\big\{(\eta_j(X'_s)-\eta_j(X'_{s-}))(\eta_j(X''_s)-\eta_j(X''_{s-}))\big\} .
\end{align*}
Then, by Lemma \ref{conditionalcopylemma} (4), we have
\begin{align*}
J_2=&\,\frac{1}{2}\sum_{j,k}\int_{s}^{t}  \,\partial_{z_j z_k}f(r-,\mathbf{Z}_{r-})d[Z^j,Z^k]_r
-\frac{1}{2}\sum_{s<r\leq t}\bigg\{ \sum_{j,k}\partial_{z_j z_k}f(r-,\mathbf{Z}_{r-})(\Delta Z^j_r\Delta Z^k_s) \bigg\}\\
=&\,\overline{\mathbb{E}}\bigg[ \frac{1}{2}\sum_{j,k}\int_{s}^{t}  \,\partial_{z_j,z_k}f(r-,\mathbf{Z}_{r-})\partial_x\eta_j(X'_{r-})(\partial_x\eta_k(X''_{r-}))^{\intercal  }:d[X',X'']^c_r \\
&\quad+\frac{1}{2}\sum_{s<r\leq t}\Big( \sum_{j,k}\partial_{z_j z_k}f(r-,\mathbf{Z}_{r-})(\eta_j(X'_s)-\eta_j(X'_{s-}))(\eta_k(X''_s)-\eta_k(X''_{s-})) \Big) \Big|\mathscr{F} \bigg] \\
&-\frac{1}{2}\sum_{s<r\leq t}\overline{\mathbb{E}}\bigg[ \sum_{j,k}\partial_{z_j z_k}f(r-,\mathbf{Z}_{r-})(\eta_j(X'_s)-\eta_j(X'_{s-}))(\eta_k(X''_s)-\eta_k(X''_{s-}))\Big|\mathscr{F} \bigg]\\
=&\,\overline{\mathbb{E}}\left[ \frac{1}{2}\int_{s}^{t}  \,\partial_{\mu\mu}F(r-,\mu_{r-},X'_{r-},X''_{r-}):d[X',X'']^c_r\Big|\mathscr{F} \right]\\
&+\overline{\mathbb{E}}\bigg[ \frac{1}{2}\sum_{s<r\leq t}\bigg\{ \frac{\delta^2F}{\delta\mu^2}(r-,\mu_{r-},X'_r,X''_r)-\frac{\delta^2F}{\delta\mu^2}(r-,\mu_{r-},X'_{r-},X''_r)  \\
&\qquad-\frac{\delta^2F}{\delta\mu^2}(r-,\mu_{r-},X'_r,X''_{r-})+\frac{\delta^2F}{\delta\mu^2}(r-,\mu_{r-},X'_{r-},X''_{r-}) \bigg\}\Big|\mathscr{F} \bigg]\\
&-\frac{1}{2}\sum_{s<r\leq t}\overline{\mathbb{E}}\bigg[ \frac{\delta^2F}{\delta\mu^2}(r-,\mu_{r-},X'_r,X''_r)-\frac{\delta^2F}{\delta\mu^2}(r-,\mu_{r-},X'_{r-},X''_r)  \\
&\qquad-\frac{\delta^2F}{\delta\mu^2}(r-,\mu_{r-},X'_r,X''_{r-})+\frac{\delta^2F}{\delta\mu^2}(r-,\mu_{r-},X'_{r-},X''_{r-}) \Big|\mathscr{F}\bigg] .
\end{align*}
The last term of the second equality above follows from
\begin{align}\label{deltaZ}
\Delta Z^j_t=&\,\overline{\mathbb{E}}[\eta_j(X_t)|\mathscr{G}_T]-\lim_{s\to t}\overline{\mathbb{E}}[\eta_j(X_s)|\mathscr{G}_T]=\overline{\mathbb{E}}[\eta_j(X_t)|\mathscr{G}_T]-\overline{\mathbb{E}}[\lim_{s\to t}\eta_j(X_s)|\mathscr{G}_T]\nonumber\\
=&\,\overline{\mathbb{E}}[\eta_j(X_t)-\eta_j(X_{t-})|\mathscr{G}_T]=\overline{\mathbb{E}}[\eta_j(X'_t)-\eta_j(X'_{t-})|\mathscr{F}], \ \ \ a.s.,
\end{align}
where the second equality of (\ref{deltaZ}) holds due to the conditional dominated convergence theorem. Still by the same argument as Theorem \ref{itowentzellc1thm} (see \eqref{eqsum}-\eqref{eq sum explain2}), $J_2$ is rewritten as
\begin{align*}
J_2&=\overline{\mathbb{E}}\bigg[ \frac{1}{2}\int_{s}^{t}  \,\partial_{\mu\mu}F(r-,\mu_{r-},X'_{r-},X''_{r-}):d[X',X'']^c_r \\
&\qquad+\frac{1}{2}\sum_{s<r\leq t}\bigg\{ \frac{\delta^2F}{\delta\mu^2}(r-,\mu_{r-},X'_r,X''_r)-\frac{\delta^2F}{\delta\mu^2}(r-,\mu_{r-},X'_{r-},X''_r)   \\
&\qquad\qquad-\frac{\delta^2F}{\delta\mu^2}(r-,\mu_{r-},X'_r,X''_{r-})+\frac{\delta^2F}{\delta\mu^2}(r-,\mu_{r-},X'_{r-},X''_{r-}) \bigg\}\mathbf{1}_{\{\mu_r=\mu_{r-}\}}\Big|\mathscr{F} \bigg].
\end{align*}
Applying the same method as above to $J_3$, we have
\begin{align*}
J_3=&\sum_j\int_{s}^{t}  \,\partial_{z_j}h_r(\mathbf{Z}_{r-})d[Z^j,Y]_r-\sum_j\sum_{s<r\leq t}\{\partial_{z_j}h_r(\mathbf{Z}_{r-})(\Delta Z^j_r\Delta Y_r)\} \\
=\,&\overline{\mathbb{E}}\bigg[ \sum_j\int_{s}^{t}  \,\partial_{z_j}h_r(\mathbf{Z}_{r-})\partial_x\eta_j(X_{r-}'):d[X',Y]_r^c +\sum_{s<r\leq t}\{(\eta_j(X'_r)-\eta_j(X'_{r-}))(Y_r-Y_{r-})\}\Big|\mathscr{F} \bigg]\\
&-\sum_{s<r\leq t}\bigg\{ \sum_j\partial_{z_j}h_r(\mathbf{Z}_{r-})(\eta_j(X'_r)-\eta_j(X'_{r-}))(Y_r-Y_{r-}) \bigg\}\\
=\,&\overline{\mathbb{E}}\bigg[ \int_{s}^{t}  \,\partial_{\mu}H_r(\mu_{r-},X'_{r-}):d[X',Y]^c_r+\sum_{s<r\leq t}(\frac{\delta H_r}{\delta\mu}(\mu_{r-},X_r)-\frac{\delta H_r}{\delta\mu}(\mu_{r-},X_{r-}))\Delta Y_r \Big|\mathscr{F}\bigg]\\
&-\sum_{s<r\leq t}\bigg\{\overline{\mathbb{E}}\bigg[ \frac{\delta H_r}{\delta\mu}(\mu_{r-},X'_r)-\frac{\delta H_r}{\delta\mu}(\mu_{r-},X'_{r-})\Delta Y_r\Big|\mathscr{F}\bigg]\bigg\},
\end{align*}
which can be rewritten as 
\begin{align*}
\overline{\mathbb{E}}\bigg[ \int_{s}^{t}  \,\partial_{\mu}H_r(\mu_{r-},X'_{r-}):d[X',Y]^c_r+\sum_{s<r\leq t}(\frac{\delta H_r}{\delta\mu}(\mu_{r-},X'_r)-\frac{\delta H_r}{\delta\mu}(\mu_{r-},X'_{r-}))\mathbf{1}_{\{\mu_r=\mu_{r-}\}}\Delta Y_r \Big|\mathscr{F}  \bigg].
\end{align*}
For $J_4$, by \eqref{directcalculus}, we can easily get 
\begin{align*}
J_4=&\sum_{s<r\leq t}\bigg\{ F(r,\mu_r)-F(r-,\mu_{r-})-H_r(\mu_{r-})\Delta Y_r\\&\qquad\qquad-\overline{\mathbb{E}}\left[\frac{\delta F}{\delta\mu}(r-,\mu_{r-},X'_r)-\frac{\delta F}{\delta\mu}(r-,\mu_{r-},X'_{r-})\Big|\mathscr{F}\right]\bigg\}.
\end{align*}
Thus, the It{\^o}-Wentzell-Lions formula for $\mathcal{C}^2-$cylindrical functions is established.\\
{\bf{Step 2 and Step 3}} Localization argument and cylindrical approximation.\\
Thanks to the fact that Proposition \ref{cylindricalapproximationprop} still holds for RF-Fully-$\mathcal{C}^2$ functions (see Theorem 4.10 in \cite{cox2024controlled} or Proposition 4.2 in \cite{guo2024ito}), we can proceed these two steps exactly the same as those of Theorem \ref{itowentzellc1thm}. Note that the sequence of approximating functions $H^n$ in Step 3 of the proof of Theorem \ref{itowentzellc1thm} can be chosen uniformly bounded, as well as its derivatives, hence, Assumption \ref{assumption h bounded} is verified.
We omit the rest of the proof.
\end{proof}

\section{Some special cases}
\subsection{It{\^o}-Wentzell-Lions formula for time-space-measure-dependent functions}
By using chain rule, we can extend It{\^o}-Wentzell-Lions formula \eqref{itowentzellc1} to the functions involving an extra variable in $\mathbb{R}^d$. Precisely, we consider a function $F:\Omega\times[0,T]\times \mathbb{R}^d\times\mathcal{P}_2(\mathbb{R}^d)\to\mathbb{R}$.
\begin{definition}
$F$ is said to be RF-Partially-${C}^{2,2}(\mathbb{R}^d\times\mathcal{P}_2(\mathbb{R}^d))$ if\\
(i) $F(\cdot, x,\cdot)$ is $RF-Partially-\cC^2$ for any $x\in \mathbb{R}^d$;\\
(ii) $\partial_x F(t,x,\mu)$ and $\partial_{xx}F(t,x,\mu)$ exist for any $(t,x,\mu)$, moreover, $\partial_xF$ and $\partial_{xx}F$ are jointly continuous w.r.t. $(t,x,\mu)$.
\end{definition}
\begin{corollary}\label{corollary1}
Let $X,Y$ be semimartingales satisfying Assumption 1, and denote $\mu_t:=\mathbb{P}_{X_t}$. Then for any RF-Partially-${C}^{2,2}(\mathbb{R}^d\times\mathcal{P}_2(\mathbb{R}^d))$ function $F$ satisfying Assumption 2, it holds for any $0\leq s\leq t\leq T$ that
\begin{align}\label{eq coro1}
&F(t,X_t,\mu_t)-F(s,X_s,\mu_s)=\int_{s}^{t}  \,G_r(X_{r-},\mu_{r-})dr +\int_{s}^{t}  \,H_r(X_{r-},\mu_{r-})dY_r\nonumber\\
&+\tilde{\mathbb{E}} \left[  \int_{s}^{t}  \,\partial_{\mu}F(r-,X_{r-},\mu_{r-},\tilde{X} _{r-}) d\tilde{X} _r+\frac{1}{2}\partial_x\partial_{\mu}F(r-,X_{r-},\mu_{r-},\tilde{X} _{r-}):d[\tilde{X} ,\tilde{X} ]_r^c \right]\nonumber\\
&+\sum_{s<r\leq t}\left\{F(r,X_{r-},\mu_r)-F(r-,X_{r-},\mu_{r-})-H_r(X_{r-},\mu_{r-})\Delta Y_r\right\}\nonumber\\
&+\tilde{\mathbb{E}} \bigg[ \sum_{s<r\leq t}\left\{ (\frac{\delta F}{\delta\mu}(r-,X_{r-},\mu_{r-},\tilde{X} _r)-\frac{\delta F}{\delta\mu}(r-,X_{r-},\mu_{r-},\tilde{X} _{r-}))\mathbf{1}_{\{\mu_r=\mu_{r-}\}}\right\}\nonumber\\
&\qquad-\sum_{s<r\leq t}\left\{ \partial_{\mu}F(r-,\mu_{r-},\tilde{X} _{r-})\Delta \tilde{X} _r\right\}\bigg]
+\int_{s}^{t}  \,\partial_{x}F(r,X_{r-},\mu_{r})dX_r\nonumber\\&+\frac{1}{2}\int_{s}^{t}  \,\partial_{xx}F(r,X_{r-},\mu_{r}):d[X,X]^c_r+\int_{s}^{t}  \,\partial_{x}H_r(X_{r-},\mu_{r}):d[X,Y]^c_r\nonumber\\&+
\sum_{s<r\leq t}\left\{ F(r,X_r,\mu_r)-F(r,X_{r-},\mu_{r})-\sum_{i=1}^d\partial_{x_i}F(r,X_{r-},\mu_{r})\Delta X_r \right\},\ \ \ a.s.
\end{align}
\end{corollary}
To prove Corollary \ref{corollary1}, we need the following lemma which can be directly obtained by the standard It{\^o} formula.
\begin{lemma}\label{lemma for C1}
Let $F,\ X,\ Y$ and $\mu$ be the same as in Corollary 1, then for any $0\leq s<t\leq T$ the following relation holds almost surely
\begin{align}\label{111}
F(t,X_t,\mu_t)-F(t,X_s,\mu_t)
=& \int_{s}^{t}\partial_x F(t,X_{r-},\mu_{t})dX_r+\int_{s}^{t}\partial_x H_{r}(X_{r-},\mu_{t}):d[X,Y]_r^c\nonumber\\&+\frac{1}{2}\int_{s}^{t}\, \partial_{xx}F(t,X_{r-},\mu_{t}):d[X,X]^c_r\\&+\sum_{s<r\leq t}\Big( F(t,X_r,\mu_{t})-F(t,X_{r-},\mu_{t})-\partial_x F(t,X_{r-},\mu_{t})\Delta X_r\Big).\nonumber
\end{align}
\end{lemma}
\begin{proof}
For $u\in[s,t]$, define $\Phi(X_u):=F(t,X_{u},\mu_{t})$ and note that
\begin{align*}
\Phi(X_u):=F(t,X_{u},\mu_{t})
=F_0(X_u,\mu_{t})+\int_{0}^{t}  \,G_r(X_u,\mu_{t})dr +\int_{0}^{t}  \,H_r(X_u,\mu_{t})dY_r 
\end{align*}
is a semimartingale. Then applying the standard It{\^o} formula to $\Phi(X_u)$ yields
\begin{align*}
\Phi(X_t)-\Phi(X_s)=&\,F(t,X_{t},\mu_{t})-F(t,X_{s},\mu_{t})\\
=&\,\int_{s}^{t}\partial_x F_0(X_{r-},\mu_{t})dX_r^c +\frac{1}{2}\int_{s}^{t}  \,\partial_{xx}F_0(X_{r-},\mu_{t}):d[X,X]^c_r\\
&+\int_{s}^{t}\partial_x G_r(X_{r-},\mu_{t}) \,dr dX_r^c +\frac{1}{2}\int_{s}^{t}  \,\partial_{xx}G_r(X_{r-},\mu_{t}) \,dr :d[X,X]^c_r\\
&+\int_{s}^{t}\partial_x H_r(X_{r-},\mu_{t}) \,dY_r^c dX_r^c +\frac{1}{2}\int_{s}^{t}  \,\partial_{xx}H_r(X_{r-},\mu_{t}) \,dY_r^c :d[X,X]^c_r\\
&+\int_{s}^{t}\partial_x H_r(X_{r-},\mu_{t}) :d[X,Y]^c_r+\sum_{s<r\leq t}\left( F(t,X_r,\mu_{t})-F(t,X_{r-},\mu_{t})\right),
\end{align*}
which coincides with \eqref{111}. \end{proof}
Now we sketch the proof of Corollary 1.\\
\textit{Proof of Corollary 1.} To use the argument of the chain rule, we fix an $m\in \mathbb{N}$ and define the partition $\pi_{s,t}^m:=\{s=t^m_0<\cdots<t^m_{m}=t\}$, such that, when $m\to\infty$,
\begin{equation*}
\max_{j}|t^m_j-t^m_{j+1}|\to 0,\ \ \ \max_j\sup_{r\in(t^m_j,t^m_{j+1}]}\left\lvert Y_{r-}-Y_{t^m_j}\right\rvert\to 0.
\end{equation*}
It is clear that
\begin{align}\label{C1 eq1}
F(t,X_t,\mu_t)-F(s,X_s,\mu_s)=&\sum_{j=0}^m\{F(t^m_{j+1},X_{t^m_{j+1}},\mu_{t^m_{j+1}})-F(t^m_j,X_{t^m_j},\mu_{t^m_j}) \}\nonumber\\=&\,\sum_{j=0}^m\{ F(t^m_{j+1},X_{t^m_{j+1}},\mu_{t^m_{j+1}})-F(t^m_{j+1},X_{t^m_j},\mu_{t^m_{j+1}})\}\nonumber\\&+\sum_{j=0}^m\{F(t^m_{j+1},X_{t^m_{j}},\mu_{t^m_{j+1}})-F(t^m_j,X_{t^m_j},\mu_{t^m_j})\}.
\end{align}
For the first term on the right hand side of \eqref{C1 eq1}, by Lemma \ref{lemma for C1}, it holds almost surely
\begin{equation*}\begin{split}
\sum_{j=0}^m&\big\{ F(t^m_{j+1},X_{t^m_{j+1}},\mu_{t^m_{j+1}})-F(t^m_{j+1},X_{t^m_j},\mu_{t^m_{j+1}})\big\}=\sum_{j=0}^m\Bigg\{ \int_{t^m_j}^{t^m_{j+1}}\partial_x F(t^m_{j+1},X_{r-},\mu_{t^m_{j+1}})dX_r\\&+\frac{1}{2}\int_{t^m_j}^{t^m_{j+1}} \partial_{xx}F(t^m_{j+1},X_{r-},\mu_{t^m_{j+1}}):d[X,X]^c_r+\int_{t^m_j}^{t^m_{j+1}}\,\partial_x H_{r}(X_{r-},\mu_{t^m_{j+1}}):d[X,Y]_r^c\\&+\sum_{t^m_j<r\leq t^m_{j+1}}\Big( F(t^m_{j+1},X_r,\mu_{t^m_{j+1}})-F(t^m_{j+1},X_{r-},\mu_{t^m_{j+1}})
-\partial_x F(t^m_{j+1},X_{r-},\mu_{t^m_{j+1}})\Delta X_r\Big)\Bigg\}.
\end{split}\end{equation*}
Then, let $m\to\infty$, we have
\begin{equation}\label{C1 eq2}\begin{split}
\lim_{m\to\infty}\sum_{j=0}^m&\{ F(t^m_{j+1},X_{t^m_{j+1}},\mu_{t^m_{j+1}})-F(t^m_{j+1},X_{t^m_j},\mu_{t^m_{j+1}})\}=\int_{s}^{t}\partial_x F(r,X_{r-},\mu_{r})dX_r\\&+\frac{1}{2}\int_{s}^{t} \partial_{xx}F(r,X_{r-},\mu_{r}):d[X,X]^c_r+\int_{s}^{t}\partial_x H_{r}(X_{r-},\mu_{r}):d[X,Y]_r^c\\&+\sum_{s<r\leq t}\Big( F(r,X_r,\mu_{r})-F(r,X_{r-},\mu_{r})-\partial_x F(r,X_{r-},\mu_{r})\Delta X_r\Big)\Bigg\}, \ \  a.s.
\end{split}\end{equation}
For the second term, by using Theorem \ref{itowentzellc1thm} and then taking limit, we get
\begin{align}\label{secondtermin22}
&\lim_{m\to\infty}\sum_{j=0}^m\{F(t^m_{j+1},X_{t^m_{j}},\mu_{t^m_{j+1}})-F(t^m_j,X_{t^m_j},\mu_{t^m_j})\}\nonumber\\=&\int_{s}^{t}  \,G_r(X_{r-},\mu_{r-})dr +\int_{s}^{t}  \,H_r(X_{r-},\mu_{r-})dY_r\nonumber\\
&+\tilde{\mathbb{E}} \left[\int_{s}^{t}  \,\partial_{\mu}F(r-,X_{r-},\mu_{r-},\tilde{X} _{r-}) d\tilde{X} _r+\frac{1}{2}\partial_x\partial_{\mu}F(r-,X_{r-},\mu_{r-},\tilde{X} _{r-}):d[\tilde{X} ,\tilde{X} ]_r^c \right]\nonumber\\
&+\sum_{s<r\leq t}\left\{F(r,X_{r-},\mu_r)-F(r-,X_{r-},\mu_{r-})-H_r(X_{r-},\mu_{r-})\Delta Y_r\right\}\nonumber\\
&+\tilde{\mathbb{E}} \bigg[ \sum_{s<r\leq t}\left\{ (\frac{\delta F}{\delta\mu}(r-,X_{r-},\mu_{r-},\tilde{X} _r)-\frac{\delta F}{\delta\mu}(r-,X_{r-},\mu_{r-},\tilde{X} _{r-}))\mathbf{1}_{\{\mu_r=\mu_{r-}\}}\right\}\nonumber\\
&\qquad-\sum_{s<r\leq t}\left\{ \partial_{\mu}F(r-,X_{r-},\mu_{r-},\tilde{X} _{r-})\Delta \overline{X} _r\right\}\bigg],\ \ a.s. 
\end{align}
The desired formula \eqref{eq coro1} is obtained by combining \eqref{C1 eq2} and \eqref{secondtermin22}.$\hspace{4cm}\Box$

\begin{remark}
The readers may observe that the chain rule can be obtained in another way:
\begin{align}\label{eq coro chain2}
F(t,X_t,\mu_t)-F(s,X_s,\mu_s)=&\sum_{j=1}^m\{F(t^m_{j+1},X_{t^m_{j+1}},\mu_{t^m_{j+1}})-F(t^m_j,X_{t^m_j},\mu_{t^m_j}) \}\nonumber\\
=&\sum_{j=0}^m\{ F(t^m_{j+1},X_{t^m_{j+1}},\mu_{t^m_{j+1}})-F(t^m_{j+1},X_{t^m_{j+1}},\mu_{t^m_{j}})\}\nonumber\\&
+\sum_{j=0}^m\{F(t^m_{j+1},X_{t^m_{j+1}},\mu_{t^m_{j}})-F(t^m_j,X_{t^m_j},\mu_{t^m_j})\}.
\end{align}
For the first term on the right hand side of \eqref{eq coro chain2}, due to Theorem 3.1 in \cite{guo2023ito}, the following relation holds almost surely
\begin{align*}
&F(t^m_{j+1},X_{t^m_{j+1}},\mu_{t^m_{j+1}})-F(t^m_{j+1},X_{t^m_{j+1}},\mu_{t^m_{j}})\\=&\,\tilde{\mathbb{E}} \bigg[\int_{t^m_j}^{t^m_{j+1}}  \,\partial_{\mu}F(t^m_{j+1},X_{t^m_{j+1}},\mu_{r-},\tilde{X}_{r-} )d\tilde{X}^c_r+\frac{1}{2}\int_{t^m_j}^{t^m_{j+1}}  \,\partial_x\partial_{\mu}F(t^m_{j+1},X_{t^m_{j+1}},\mu_{r-},\tilde{X}_{r-} ):d[\tilde{X},\tilde{X}  ]^c_r\\
&\hspace{0.5cm}+\sum_{t^m_j<r\leq t^m_{j+1}}\bigg\{ F(t^m_{j+1},X_{t^m_{j+1}},\mu_{r})-F(t^m_{j+1},X_{t^m_{j+1}},\mu_{r-})\\
&\hspace{2.9cm}+\Big(\frac{\delta F}{\delta\mu}(t^m_{j+1},X_{t^m_{j+1}},\mu_{r},\tilde{X}_r )-\frac{\delta F}{\delta\mu}(t^m_{j+1},X_{t^m_{j+1}},\mu_{r-},\tilde{X}_{r-} )\Big)\mathbf{1}_{\{\mu_r=\mu_{r-}\}} \bigg\}\bigg].
\end{align*}
Then, let $m\to\infty$, one has
\begin{equation*}
\begin{split}
&\lim_{m\to\infty}\sum_{j=0}^m\{ F(t^m_{j+1},X_{t^m_{j+1}},\mu_{t^m_{j+1}})-F(t^m_{j+1},X_{t^m_{j+1}},\mu_{t^m_{j}})\}\\
=&\,\tilde{\mathbb{E}} \bigg[  \int_{s}^{t}  \,\partial_{\mu}F(r-,X_r,\mu_{r-},\tilde{X} _{r-}) d\tilde{X} _r+\frac{1}{2}\partial_x\partial_{\mu}F(r-,X_r,\mu_{r-},\tilde{X} _{r-})d[\tilde{X} ,\tilde{X} ]_r^c \bigg]\\
&+\sum_{s<r\leq t}\left\{F(r,X_{r},\mu_r)-F(r-,X_{r},\mu_{r-})-H_r(X_{r},\mu_{r-})\Delta Y_r\right\}\nonumber\\
&+\tilde{\mathbb{E}} \bigg[ \sum_{s<r\leq t}\left\{ (\frac{\delta F}{\delta\mu}(r-,X_r,\mu_{r-},\tilde{X} _r)-\frac{\delta F}{\delta\mu}(r-,X_r,\mu_{r-},\tilde{X} _{r-}))\mathbf{1}_{\{\mu_r=\mu_{r-}\}}\right\} \\
&\qquad-\sum_{s<r\leq t}\{ \partial_{\mu}F(r-,X_r,\mu_{r-},\tilde{X} _{r-})\Delta \tilde{X} _r\} \bigg], \ \ \ a.s.
\end{split}
\end{equation*}
For the second term, apply Theorem \ref{standitowentzellsemimart} to $F(t,X_{t},\mu_{t^m_{j}})$ and obtain
\begin{align*}
F(t^m_{j+1}&,X_{t^m_{j+1}},\mu_{t^m_{j}})-F(t^m_j,X_{t^m_j},\mu_{t^m_j})=\int_{t_j^m}^{t^m_{j+1}}  \,G_r(X_{r-},\mu_{t^m_{j}})dr+\int_{t_j^m}^{t^m_{j+1}}  \,H_r(X_{r-},\mu_{t^m_j})dY^c_r\\
&+\frac{1}{2}\int_{t_j^m}^{t^m_{j+1}}  \,\partial_{xx}F(r-,X_{r-},\mu_{t^m_{j}}):d[X,X]^c_r+\int_{t_j^m}^{t^m_{j+1}}  \,\partial_x H_r(X_{r-},\mu_j^m):d[X,Y]^c_r\\
&+\int_{t_j^m}^{t^m_{j+1}}  \,\partial_x F(r-,X_{r-},\mu_{t^m_j})dx +\sum_{t^m_j<r\leq t^m_{j+1}}\big\{ F(r,X_r,\mu_{t^m_{j}})-F(r-,X_{r-},\mu_{t^m_{j}})\big\},  \ \ \ a.s.
\end{align*}
Taking $m\to\infty$ yields 
\begin{align*}
&\sum_{j=0}^m\{F(t^m_{j+1},X_{t^m_{j+1}},\mu_{t^m_{j}})-F(t^m_j,X_{t^m_j},\mu_{t^m_j})\}\\
\rightarrow&\int_{s}^{t}G_r(X_{r-},\mu_{r-})dr+\int_{s}^{t}  \,H_r(X_{r-},\mu_{r-})dY^c_r+\sum_{i=1}^d\int_{s}^{t}  \,\partial_{x_i}F(r-,X_{r-},\mu_{r-})dX_r^i\\&+\frac{1}{2}\sum_{i,j=1}^d\int_{s}^{t}  \partial_{x_i x_j}F(r-,X_{r-},\mu_{r-}):d[X^i,X^j]^c_r 
+\sum_{i=1}^d\int_{s}^{t}  \,\partial_{x_i}H_r(X_{r-},\mu_{r-}):d[X,Y]^c_r\\&
+\sum_{s<r\leq t}\left\{ F(r,X_r,\mu_{r-})-F(r,X_{r-},\mu_{r-})-\sum_{i=1}^d\partial_{x_i}F(r-,X_{r-},\mu_{r})\Delta X_r \right\},\ \ \ a.s.
\end{align*}
Combining the above convergence results, we derive the following formula
\begin{align}\label{eq coro1 alter}
&F(t,X_t,\mu_t)-F(s,X_s,\mu_s)=\int_{s}^{t}  \,G_r(X_{r-},\mu_{r-})dr +\int_{s}^{t}  \,H_r(X_{r-},\mu_{r-})dY_r\nonumber\\
&+\tilde{\mathbb{E}} \left[  \int_{s}^{t}  \,\partial_{\mu}F(r-,X_{r},\mu_{r-},\tilde{X} _{r-}) d\tilde{X} _r+\frac{1}{2}\partial_x\partial_{\mu}F(r-,X_{r},\mu_{r-},\tilde{X} _{r-}):d[\tilde{X} ,\tilde{X} ]_r^c \right]\nonumber\\
&+\sum_{s<r\leq t}\left\{F(r,X_{r},\mu_r)-F(r-,X_{r},\mu_{r-})-H_r(X_{r},\mu_{r-})\Delta Y_r\right\}\nonumber\\
&+\tilde{\mathbb{E}} \bigg[ \sum_{s<r\leq t}\left\{ (\frac{\delta F}{\delta\mu}(r-,X_{r},\mu_{r-},\tilde{X} _r)-\frac{\delta F}{\delta\mu}(r-,X_{r},\mu_{r-},\tilde{X} _{r-}))\mathbf{1}_{\{\mu_r=\mu_{r-}\}}\right\}\\
&\qquad-\sum_{s<r\leq t}\left\{ \partial_{\mu}F(r-,X_r,\mu_{r-},\tilde{X} _{r-})\Delta \tilde{X} _r\right\}\bigg] \nonumber\\
&+\int_{s}^{t}  \,\partial_{x}F(r-,X_{r-},\mu_{r})dX_r+\frac{1}{2}\int_{s}^{t}  \,\partial_{xx}F(r-,X_{r-},\mu_{r}):d[X,X]^c_r 
\nonumber\\&+\int_{s}^{t}  \,\partial_{x}H_r(X_{r-},\mu_{r-}):d[X,Y]^c_r\nonumber\\&+
\sum_{s<r\leq t}\left\{ F(r,X_r,\mu_{r-})-F(r,X_{r-},\mu_{r-})-\sum_{i=1}^d\partial_{x_i}F(r-,X_{r-},\mu_{r})\Delta X_r \right\},\ \ \ a.s.\nonumber
\end{align}
So far we obtain two It\^o-Wentzell-Lions formulae \eqref{eq coro1} and \eqref{eq coro1 alter}, whose difference stems from the choice of the middle term in \eqref{C1 eq1} and \eqref{eq coro chain2}. In particular, when the semimartingales $X$ and $Y$ are continuous, \eqref{eq coro1} and \eqref{eq coro1 alter} are the same and coincide with Theorem 3.9 in \cite{dos2023ito}.
\end{remark}
In the case of conditional measures, we can obtain a similar extension of It{\^o}-Wentzell-Lions formula by the same argument. 
\begin{definition}
$F$ is said to be RF-Fully-${C}^{2,2}(\mathbb{R}^d\times\mathcal{P}_2(\mathbb{R}^d))$ if\\
(i) $F(\cdot, x,\cdot)$ is $RF-Fully-\cC^2$ for any $x\in \mathbb{R}^d$;\\
(ii) $\partial_x F(t,x,\mu)$ and $\partial_{xx}F(t,x,\mu)$ exist for any $(t,x,\mu)$, moreover, $\partial_xF$ and $\partial_{xx}F$ are jointly continuous w.r.t. $(t,x,\mu)$.
\end{definition}
\begin{corollary}
Given semimartingales $X$, $Y$ satisfying Assumption 1 and a sub-filtration $\mathbb{G}:=(\mathscr{G}_t)_{t\in [0,T]}\subset\mathbb{F}$ satisfying Assumption \ref{filtrationassumption}, denote $\mu_t:=\mathbb{P}_{X_t|\mathscr{G}_t}$. 
Then, for any RF-Fully-${C}^{2,2}(\mathbb{R}^d\times\mathcal{P}_2(\mathbb{R}^d))$ function $F$ satisfying Assumption 2 and 4, it holds almost surely for any $0\leq s\leq t\leq T$ that 
\begin{align*}
&F(t,X_t,\mu_t)-F(s,X_s,\mu_s)=\int_{s}^{t}G_r(X_{r-},\mu_{r-})dr+\int_{s}^{t}H_r(X_{r-},\mu_{r-})dY_r \\&+\sum_{s<r\leq t}\left\{ F(r,X_{r-},\mu_r)-F(r-,X_{r-},\mu_{r-})- H_r(X_{r-},\mu_{r-})\Delta Y_r\right\}\\
&+\int_{s}^{t}\partial_{x}F(r,X_{r-},\mu_{r})dX_r+\frac{1}{2}\int_{s}^{t}\partial_{xx}F(r,X_{r-},\mu_{r}):d[X,X]^c_r 
+\int_{s}^{t}\partial_{x}H_r(X_{r-},\mu_{r}):d[X,Y]^c_r\\&
+\sum_{s<r\leq t}\left\{ F(r,X_r,\mu_r)-F(r,X_{r-},\mu_{r})-\partial_{x}F(r,X_{r-},\mu_{r})\Delta X_r \right\}.\\
&+\overline{\mathbb{E}}\bigg[\int_{s}^{t}\partial_{\mu}F(r-,X_{r-},\mu_{r-},X'_{r-})dX'_r+\frac{1}{2}\partial_{\mu}\partial_x F(r-,X_{r-},\mu_{r-},X'_{r-}):d[X',X']^c_r \\
&\quad+\frac{1}{2}\int_{s}^{t}\partial_{\mu\mu}F(r-,X_{r-},\mu_{r-},X'_{r-},X''_{r-}):d[X',X'']^c_r+\int_{s}^{t}\partial_{\mu}H_r(X_{r-},\mu_{r-},X'_{r-}):d[X',Y]^c_r\\
&\quad+\sum_{s<r\leq t}\bigg( \frac{1}{2} \big( \frac{\delta^2F}{\delta\mu^2}(r-,X_{r-},\mu_{r-},X'_r,X''_r)-\frac{\delta^2F}{\delta\mu^2}(r-,X_{r-},\mu_{r-},X'_{r-},X''_r)  \\
&\qquad\qquad-\frac{\delta^2F}{\delta\mu^2}(r-,X_{r-},\mu_{r-},X'_r,X''_{r-})+\frac{\delta^2F}{\delta\mu^2}(r-,X_{r-},\mu_{r-},X'_{r-},X''_{r-})\big)\\
&\qquad+(\frac{\delta H_r}{\delta\mu}(X_{r-},\mu_{r-},X'_r)-\frac{\delta H_r}{\delta\mu}(X_{r-},\mu_{r-},X'_{r-}))\Delta Y_r\\
&\qquad+\frac{\delta F}{\delta\mu}(r-,X_{r-},\mu_{r-},X'_r)-\frac{\delta F}{\delta\mu}(r-,X_{r-},\mu_{r-},X'_{r-})\bigg)\mathbf{1}_{\{\mu_r=\mu_{r-}\}}\\
&\quad-\partial_{\mu}F(r-,X_{r-},\mu_{r-},X'_{r-})\Delta X'_r  \Big|\mathscr{F}\bigg].
\end{align*}
\end{corollary}

\subsection{Case of Poisson random measure}
In this subsection, we give the specific It{\^o}-Wentzell-Lions formula in the case of Poisson random measures, i.e., the random field is driven by Poisson random measure and so is the state process.
Denote $(E,\mathcal{E})$ a measurable space and $N(ds,de)$ an $\mathbb{F}$-adapted Poisson random measure on $([0,T]\times E,\mathcal{B}([0,T])\otimes\mathcal{E})$. 
Moreover, denote $\nu$ the finite intensity measure of $N$.
The compensated Poisson random measure is defined by $\widetilde{N}(ds,de):=N(ds,de)-\nu(de)ds$ (see \cite{applebaum2009levy,rudiger2004stochastic}, for instance). Let $b:=(b_t)_{t\in [0,T]}$ and $\sigma:=(\sigma_t)_{t\in [0,T]}$ be $\mathbb{R}^d$ and $\mathbb{R}^{d\times d}$ valued $\mathbb{F}$-adapted processes, respectively, and $\beta(e):=(\beta_t(e))_{t\in[0,T]}$ be $\mathbb{R}^d$-valued $\mathbb{F}$-predictable process for any $e\in E$. Consider the following state process driven by jump-diffusion process,
\begin{eqnarray}\label{poissonjump}
dX_t=b_t dt+\sigma_t dW_t+\int_{E}\,\beta_t(e)N(dt,de),
\end{eqnarray}
where $W$ is a standard Brownian motion independent of $N$.
It is easy to check that $(X_t)_{t\in [0,T]}$ satisfies Assumption 1 if
\begin{eqnarray*}
\mathbb{E}\left[ \int_{0}^{T}  \,\left(\left\lvert b_s\right\rvert^2+\left\lvert \sigma_s\right\rvert^2+\int_{E} \,\left\lvert \beta_s\right\rvert^2 \nu(de) \right)  ds  \right]<\infty.
\end{eqnarray*}
It is well known that the measure-valued process $(\mathbb{P}_{X_t})_{t\in[0,T]}$ is continuous in time (see \cite{gaviraghi2017theoretical}).
\begin{corollary}\label{specific-full-poisson}
Given $X$ defined in \eqref{poissonjump} and RF-Partially-$\mathcal{C}^2$ function $F:\Omega\times[0,T]\times\mathcal{P}_2(\mathbb{R}^d)\rightarrow\mathbb{R}$ admitting the following expression
\begin{eqnarray*}
F(t,\mu)=F_0(\mu)+\int_{0}^{t}  \,G_s(\mu)ds +\int_{0}^{t}  \,H_s(\mu)dW_s+\int_{0}^{t}  \,\int_{E} \,J_s(e,\mu)N(de,ds). 
\end{eqnarray*}
Denote $\mu_t:=\mathbb{P}_{X_t}$, then, for any $0\leq s\leq t\leq T$, we have
\begin{align*}
&F(t,\mu_t)-F(s,\mu_s)=\int_{s}^{t}  \,G_r(\mu_{r-})dr+\int_{s}^{t}  \,H_r(\mu_{r-})dW_r+\int_{s}^{t}  \,\int_{E} \,J_r(e,\mu_r)\nu(de,dr) \\
&\qquad+\tilde{\mathbb{E}}\bigg[ \int_{s}^{t}  \,(\partial_{\mu}F(r-,\mu_{r-},\tilde{X}_{r-})b_r+\frac{1}{2}\partial_x\partial_{\mu}F(r-,\mu_{r-},\tilde{X}_{r-}):\sigma_r\sigma_r^{\intercal  })dr \\
&\qquad\qquad+\int_{s}^{t}  \,\int_{E} \,\frac{\delta F}{\delta\mu}(r-,\mu_r,\tilde{X}_r+\beta_r(e))-\frac{\delta F}{\delta\mu}(r-,\mu_r,\tilde{X}_r)\nu(de,dr)  \bigg],  \ \ \ \mathbb{P}-a.s.
\end{align*}
\end{corollary}
\begin{proof}
Apply Theorem \ref{itowentzellc1thm} to $F(t,\mu_t)$ and rewrite it as
\begin{align*}
F(t,\mu_t)-F(s,\mu_s)&=\int_{s}^{t}G_r(\mu_{r-})dr +\int_{s}^{t}H_r(\mu_{r-})dY^c_r+\sum_{s<r\leq t}\left\{F(r,\mu_r)-F(r-,\mu_{r-})\right\}\\
&+\tilde{\mathbb{E}}\left[\int_{s}^{t}\partial_{\mu}F(r-,\mu_{r-},\tilde{X}_{r-}) d\tilde{X}^c_r+\frac{1}{2}\partial_x\partial_{\mu}F(r-,\mu_{r-},\tilde{X}_{r-}):d[\tilde{X},\tilde{X}]_r^c \right]\\
&+\tilde{\mathbb{E}}\left[ \sum_{s<r\leq t}\left\{ (\frac{\delta F}{\delta\mu}(r-,\mu_{r-},\tilde{X}_r)-\frac{\delta F}{\delta\mu}(r-,\mu_{r-},\tilde{X}_{r-}))\mathbf{1}_{\{\mu_r=\mu_{r-}\}}\right\} \right], \ \ \ a.s.
\end{align*}
By the boundedness of $\partial_{\mu}F$, we can get rid of the martingale part of $\tilde{X}^c$, that is
\begin{eqnarray*}&&
\tilde{\mathbb{E}}\left[\int_{s}^{t}\partial_{\mu}F(r-,\mu_{r-},\tilde{X}_{r-}) d\tilde{X}^c_r+\frac{1}{2}\partial_x\partial_{\mu}F(r-,\mu_{r-},\tilde{X}_{r-}):d[\tilde{X},\tilde{X}]_r^c \right]\\&&
=\tilde{\mathbb{E}}\bigg[ \int_{s}^{t}\Big(\partial_{\mu}F(r-,\mu_{r-},\tilde{X}_{r-})b_r+\frac{1}{2}\partial_x\partial_{\mu}F(r-,\mu_{r-},\tilde{X}_{r-}):\sigma_r\sigma_r^{\intercal  }\Big)dr\bigg].
\end{eqnarray*}
Then, note that at time $r$, $F(r-,\mu_r)$ and $X_{r-}$ have jumps $J_r(e,\mu_r)$ and $\beta_r(e)$ with Poisson random measure $N(de,dr)$, respectively. Thus
$$\sum_{s<r\leq t}\left\{F(r,\mu_r)-F(r-,\mu_{r-})\right\}=\sum_{s<r\leq t}\left\{F(r,\mu_r)-F(r-,\mu_{r})\right\}=\int_{s}^{t}\int_{E}J_r(e,\mu_r)\nu(de,dr)$$
and
\begin{eqnarray*}&&
\tilde{\mathbb{E}}\bigg[ \sum_{s<r\leq t}\left\{ (\frac{\delta F}{\delta\mu}(r-,\mu_{r-},\tilde{X}_r)-\frac{\delta F}{\delta\mu}(r-,\mu_{r-},\tilde{X}_{r-}))\mathbf{1}_{\{\mu_r=\mu_{r-}\}}\right\} \bigg]\\&&
=\tilde{\mathbb{E}}\left[ \int_{s}^{t}\int_{E} \,\left(\frac{\delta F}{\delta\mu}(r-,\mu_r,\tilde{X}_r+\beta_r(e))-\frac{\delta F}{\delta\mu}(r-,\mu_r,\tilde{X}_r)\right)\nu(de,dr)\right].
\end{eqnarray*}
The proof is complete.
\end{proof}
Finally, we give the specific It\^o-Wentzell-Lions formula for the flows of conditional measures in the case of Poisson random measures.
\begin{corollary}\label{specific-conditional-poisson}
Consider a semimartingale $X$ defined in \eqref{poissonjump} and a RF-Partially-$\mathcal{C}^2$ function $F:\Omega\times[0,T]\times\mathcal{P}_2(\mathbb{R}^d)\rightarrow\mathbb{R}$ satisfying Assumption 4. Given a sub-$\sigma$-algebra $\mathbb{G}:=(\mathscr{G}_t)_{t\in[0,T]}\subset\mathbb{F}$ satisfying Assumption 3, denote $\mu_t:=\mathbb{P}_{X_t|\mathscr{G}_t}$. Let the enlarged space $(\overline{\Omega},\overline{\mathscr{F}},\overline{\mathbb{F}},\overline{\mathbb{P}})$ be the same as in Section 4.
Then, for any $0\leq s\leq t\leq T$, we have
\begin{align*}
&F(t,\mu_t)-F(s,\mu_s)=\int_{s}^{t}  \,G_r(\mu_{r-})dr+\int_{s}^{t}  \,H_r(\mu_{r-})dW_r+\int_{s}^{t}  \,\int_{E} \,J_r(e,\mu_r)\nu(de,dr) \\
&+\mathbb{\overline{E}}\Bigg[ \int_{s}^{t}  \,(\partial_{\mu}F(r-,\mu_{r-},X_{r-}')b'_r+\frac{1}{2}\partial_x\partial_{\mu}F(r-,\mu_{r-},X_{r-}'):\sigma_r'\sigma_r^{'\intercal})dr\\
&\qquad+\int_{s}^{t}  \,\frac{1}{2}\partial_{\mu}\partial_{\mu}F(r-,\mu_{r-},X_{r-}',X_{r-}'')\sigma_r'\sigma_r^{''\intercal  }:d[W',W'']_r\\
&\qquad+\int_{s}^{t}  \,\partial_{\mu}F(r-,\mu_{r-},X_{r-}')\sigma_r'dW_r'+\int_{s}^{t}\,\partial_{\mu}H_r(\mu_r,X_{r-}')\sigma_r':d[Y,W']_r\\
&\qquad+\int_{s}^{t}  \,\int_{E} \,\bigg( \frac{1}{2}(\frac{\delta^2F}{\delta\mu^2}(r-,\mu_r,X_r'+\beta_r'(e),X_r''+\beta_r''(e))\\
&\qquad\qquad-\frac{\delta^2F}{\delta\mu^2}(r-,\mu_r,X_r',X_r''+\beta_r''(e))-\frac{\delta^2F}{\delta\mu^2}(r-,\mu_r,X_r'+\beta_r'(e),X_r'') )\\
&\qquad\qquad+\frac{\delta^2F}{\delta\mu^2}(r-,\mu_r,X_r',X_r'')) +(\frac{\delta H_r}{\delta\mu}(\mu_r,X'_r+\beta_r'(e))-\frac{\delta H_r}{\delta\mu}(\mu_r,X'_r))J_r(e,\mu_r)\\
&\qquad\qquad+\frac{\delta F}{\delta\mu}(r-,\mu_r,X_r'+\beta'_r(e))- \frac{\delta F}{\delta\mu}(r-,\mu_r,X_r')\bigg)\nu(de,dr)\Big| \mathscr{F}\Bigg],\ \ \ \mathbb{P}-a.s.
\end{align*}
\end{corollary}
The proof of Corollary \ref{specific-conditional-poisson} is the same as that of Corollary \ref{specific-full-poisson}, except that we need to note that the martingale term can not be removed when taking the conditional expectation and that the variation $[dW',dW'']_\cdot$ is not $0$, because that $W'$ and $W''$ are not independent given $\mathscr{F}$.

\appendix
\section{Some useful convergence lemmas}
\begin{lemma}\label{appendixlemma1}
For fixed $T$, denote $\hat{\mathscr{F}}_T$ the $\sigma$-algebra generated by the processes $\{U_t\}_{0\leq t\leq T}$ and $\{N_t\}_{0\leq t\leq T}$, which is defined in Section 3 and satisfy Assumption 1. Then the set of random variables 
\begin{align*}
\{ \phi(U_{t_1},\ldots,U_{t_n},N_{t_1},\ldots,N_{t_n});t_i\in[0,T],\phi\in C^{\infty}_0((\mathbb{R}^d)^{2n}),n=1,2,\ldots\}
\end{align*}
is dense in $L^2(\Omega,\hat{\mathscr{F}}_T;\mathbb{R})$.
\end{lemma}
\begin{proof}
Let $\{t_i\}_{i=1}^{\infty}$ be a dense subset of $[0,T]$. For each $n$, $\mathcal{H}_n$ is the $\sigma$-algebra generated by $(U_{t_1},\ldots,U_{t_n},N_{t_1},\ldots,N_{t_n})$, and this provides $\mathcal{H}_n\subset\mathcal{H}_{n+1}$. Since $U$ is a c{\`a}dl{\`a}g process which has no jump at $T$ and $N$ is continuous, $\hat{\mathscr{F}}_T$ is the smallest $\sigma$-algebra containing all $\mathcal{H}_n$. For any $g\in L^2(\Omega,\hat{\mathscr{F}}_T;\mathbb{R})$, by the martingale convergence theorem (Corollary C.9 in \cite{oksendal2003stochastic}) we have 
\begin{align*}
\lim_{n\to\infty}\mathbb{E}[g|\mathcal{H}_n]=\mathbb{E}[g|\hat{\mathscr{F}}_T]=g, \ \ {\text in }\ L^2.
\end{align*}
Then, due to Doob-Dynkin Lemma (Lemma 2.1.2 in \cite{oksendal2003stochastic}), for any $n$, there exists a Borel measurable function $g_n:(\mathbb{R}^d)^{2n}\to\mathbb{R}$ such that
\begin{align*}
\mathbb{E}[g|\mathcal{H}_n]=g_n(U_{t_1},\ldots,U_{t_n},N_{t_1},\ldots,N_{t_n}).
\end{align*}
Furthermore, for each $n$, $g_n(U_{t_1},\ldots,U_{t_n},N_{t_1},\ldots,N_{t_n})$ can be  approximated by functions $\{\phi_{m,n}(U_{t_1},\ldots,U_{t_n},N_{t_1},\ldots,N_{t_n})\}_{m\in\mathbb{N}^*}$ as $m\to\infty$, where $\phi_{m,n}\in C^{\infty}_0((\mathbb{R}^d)^{2n})$.
\end{proof}
\begin{lemma}\label{appendixlemma2}
Let Assumption 1 and 2 hold, and the random field $F$ be defined by \eqref{decompositionFmu}. If $G:\Omega\times[0,T]\times\mathcal{P}_2(\mathbb{R}^d)\rightarrow\mathbb{R}$ is continuous in time, then for any $\eps>0$, there exists a partition $0=t_0<t_1<\ldots<t_n=T$ and a function  $g_n\in C_0^{0,1,\infty}([0,T]\times\mathcal{P}_2(\mathbb{R}^d)\times(\mathbb{R}^d)^{2n})$ for some $n\in\mathbb{N}$, such that for any $\mu$
\begin{align*}
    \mathbb{E}\bigg[ \int_0^T\, |G_s(\mu)-g_n(s,\mu,U_{s\wedge t_1},\ldots,U_{s\wedge t_n},N_{s\wedge t_1},\ldots,N_{s\wedge t_n})|^2ds \bigg]<\eps.
\end{align*}
\end{lemma}
For $H:\Omega\times[0,T]\times\mathcal{P}_2(\mathbb{R}^d)\rightarrow\mathbb{R}^l$ in \eqref{decompositionFmu}, when it is continuous in time, the same result holds. The proof of this lemma follows a similar approach to that of Lemma 4.5 in \cite{qiu2019uniqueness}, and a sketch is provided below.
\begin{proof}
First, the dominated convergence theorem indicates that for any $\mu$, $G_\cdot(\cdot,\mu)$ can be approximated in $L^2(\Omega\times[0,T])$ by random fields of the form:
\begin{align}\label{eq lemma5 time simple}
\hat{g}^l(t,\omega,\mu):=\phi_1(\omega,\mu)\mathbf{1}_{[0,t_1]}(t)+\sum_{j=2}^l\phi_j(\omega,\mu)\mathbf{1}_{(t_{j-1},t_{j}]}(t),
\end{align}
where $0=t_0<t_1<\ldots<t_n=T$, for $j=1,\ldots,l$, $\phi_j(\cdot,\cdot):=G_{t_{j-1}}(\cdot,\cdot):\Omega\times\mathcal{P}_2(\mathbb{R}^d)\to\mathbb{R}$, $\phi_j$ is $\mathcal{C}^1$ differentiable in $\mu$ and $\hat{\mathscr{F}}_{t_{j-1}}-$measurable.\\
Furthermore, since $\mathcal{P}_2(\mathbb{R}^d)$ is a Polish space, for each $j$, $\phi_j$ may be approximated monotonically (see Lemma 1.3 in \cite{da2014stochastic}, for instance) by simple random variables of the form:
\begin{align*}
\sum_{i=1}^{l_j}\mathbf{1}_{A^j_i}(\omega)h^j_i(\mu),\quad \mbox{with}\ \ h^j_i\in\mathcal{C}^1(\mathcal{P}_2(\mathbb{R}^d))),\ A^j_i\in\hat{\mathscr{F}}_{t_{j-1}},\ i=1,\ldots,l_j.
\end{align*}
Then, Lemma \ref{appendixlemma1} implies that each $\mathbf{1}_{A^j_i}(\omega)$ may be approximated in $L^2(\Omega,\hat{\mathscr{F}}_{t_{j-1}})$ by functions 
\begin{align*}
\big\{ f(U_{\hat{t}_1},\ldots,U_{\hat{t}_{l_i^j}},N_{\hat{t}_1},\ldots,N_{\hat{t}_{l_i^j}}):\hat{t}_r\in[0,t_{j-1}],\ r=1,\ldots,l^j_i,\ f_i^j\in C^{\infty}_0(\mathbb{R}^{2l^j_i})\big\}.
\end{align*}
In addition, each function $\mathbf{1}_{(t_{j-1},t_j]}$ in \eqref{eq lemma5 time simple} may be increasingly approximated by nonnegative functions $\varphi_j\in C^{\infty}((t_{j-1},T])$. Therefore, we can define the approximating functions as 
\begin{align*}
g_n(t,&\mu,U_{t\wedge \hat{t}_1},\ldots,U_{t\wedge \hat{t}_n},N_{t\wedge \hat{t}_1},\ldots,N_{t\wedge \hat{t}_n})\\
:=&\sum_{j=1}^l\sum_{i=1}^{l_j}f_i^j(U_{\hat{t}_1\wedge t_{j-1}},\ldots,U_{\hat{t}_{n}\wedge t_{j-1}},N_{\hat{t}_1\wedge t_{j-1}},\ldots,N_{\hat{t}_{n}\wedge t_{j-1}})h^j_i(\mu)\varphi_j(t),
\end{align*}
where $0=\hat{t}_0<\hat{t}_1<\ldots<\hat{t}_n=T$, $f_i^j,\varphi_j$ and $h^j_i$ are defined above. Then it is clear that the functions $\{g_n\}_{n\in\mathbb{N}^*}$ maintain $\mathcal{C}^1$-differentiability in $\mathcal{P}_2(\mathbb{R}^d)$. Thus, the required approximation is obtained.
\end{proof}
\section{Leibniz rule for It\^o integral}
In this subsection, we provide a general Leibniz rule to interchange the Fr{\'e}chet derivatives with It\^o integral. While for Leibniz rule to exchange the Fr{\'e}chet derivatives with the Lebesgue integral, the readers are referred to \cite{kammar2016note}.
\begin{lemma}\label{lemma Leibniz}
Let $H$ be a Banach space and $f$ be a progressively measurable process satisfying, for any $t\in[0,T]$, $f(t,\cdot)$ is Fr{\'e}chet differentiable at any $x\in H$ whose derivative $Df(t,x)\in H^*$ and the following integrability condition holds
\begin{align}\label{integrabilityDf}
\mathbb{E}\bigg[ \int_0^T\, \|Df(t,x)\|^2_{H^*}dt \bigg]<\infty,\quad\forall x\in H.
\end{align}
Then we have for any $x\in H$ and $h\in H$
\begin{align}\label{eq leibniz rule}
D\bigg( \int_0^T\,f(t,x)dW_t \bigg)(h)=\int_0^T\, Df(t,x)(h)dW_t,\ \ \ a.s.
\end{align}
\end{lemma}
\begin{proof}
Consider the mapping $\mathbf{T}:H\to\mathbb{R}$ defined as following
\begin{align*}
\mathbf{T}(h):=\int_0^T\, Df(t,x)(h)dW_t,\quad \forall x\in H.
\end{align*}
By the definition of Fr{\'e}chet derivatives, we have 
\begin{align*}
Df(t,x)(h):=\lim_{\eps\to 0}\frac{f(t,x+\eps h)-f(t,x)}{\eps},
\end{align*}
which is $\mathscr{F}_t-$measurable, hence the stochastic integral $\mathbf{T}(h)$ is well-defined.\\
To prove \eqref{eq leibniz rule}, we derive from the definition of Fr{\'e}chet derivatives that for any $h\in H$ satisfying $\|h\|_H=1$,
\begin{align*}
D\bigg( \int_0^T\,f(t,x)dW_t \bigg)(h)=&\lim_{\eps\to 0}\frac{1}{\eps}\bigg( \int_0^T\,f(t,x+\eps h)dW_t-\int_0^T\,f(t,x)dW_t \bigg)\\
=&\lim_{\eps\to 0}\int_0^T\, \frac{f(t,x+\eps h)-f(t,x)}{\eps}dW_t.
\end{align*}
By the arbitrariness of $h$, we only need to prove that
\begin{align*}
\int_0^T\, \frac{f(t,x+\eps h)-f(t,x)}{\eps}dW_t\to\int_0^T\,Df(t,x)(h)dW_t,\ \ \ \mbox{in} \ L^2.
\end{align*}
It follows the definition of Fr{\'e}chet derivative that
\begin{align*}
\frac{f(t,x+\eps h)-f(t,x)}{\eps}=Df(t,x)(h)+o(1).
\end{align*}
Therefore,
\begin{align*}
\lim_{\eps\to 0}\frac{f(t,x+\eps h)-f(t,x)}{\eps}=Df(t,x)(h)
\end{align*}
and
\begin{align*}
\bigg|\frac{f(t,x+\eps h)-f(t,x)}{\eps}\bigg|^2\leq \|Df(t,x)\|^2_{H^*}\|h\|^2_H+o(1).
\end{align*}
Moreover, by the integrability condition \eqref{integrabilityDf}, it holds 
\begin{align*}
\mathbb{E}\bigg[ \int_0^T\, |Df(t,x)(h)|^2dt \bigg]\leq \mathbb{E}\bigg[ \int_0^T\, \|Df(t,x)\|^2_{H^*}dt \bigg]\cdot\|h\|_H<\infty.
\end{align*}
Hence, due to the dominated convergence theorem, we obtain 
\begin{align*}
\lim_{\eps\to 0}\int_0^T\, \frac{f(t,x+\eps h)-f(t,x)}{\eps}dW_t=\int_0^T\,Df(t,x)(h)dW_t,\ \ \ \mbox{in} \ L^2.
\end{align*}
\end{proof}

%

\bibliographystyle{abbrv}
\bibliography{Ito_wentzell_lions_jump}

\end{document}